\documentclass[twoside,11pt]{article}
\usepackage{jmlr2e_draft}

\usepackage{amsmath, amssymb} 
\usepackage{stmaryrd}
\usepackage{epsfig,calc, graphicx}
\usepackage{wasysym}
\usepackage[usenames]{color}
\usepackage{url}
\usepackage{enumitem}
\usepackage[lined,algoruled,algosection]{algorithm2e}
\usepackage{ifthen}
\usepackage{bbm}
\usepackage{comment}
\newboolean{hideproof}
\setboolean{hideproof}{true}
\usepackage{subcaption}
\usepackage{hyperref}
\usepackage{array}
\usepackage{booktabs}

\newboolean{arxiv}
\setboolean{arxiv}{true}

\newcommand\ip[2]{\langle #1, #2\rangle}

\newcommand\dico{\Phi}

\newcommand\atom{\phi}

\newcommand\pdico{\Psi}

\newcommand\patom{\psi}

\newcommand\proj{R}
\newcommand\diag{\operatorname{diag}}
\newcommand\signop{\operatorname{sign}}
\newcommand\argmax{\operatorname{argmax}}
\newcommand\argmin{\operatorname{argmin}}

\newcommand{\R}{{\mathbb{R}}}

\newcommand{\E}{{\mathbb{E}}}
\newcommand{\I}{{\mathbb{I}}}

\renewcommand{\P}{{\mathbb{P}}}

\newcommand{\weights}{W}
\newcommand{\transp}{^{\ast}}
\newcommand{\Hij}{\vec{H}_{ij}}
\newcommand{\Hji}{\vec{H}_{ji}}
\newcommand{\Hjk}{\vec{H}_{jk}}
\newcommand{\Hkj}{\vec{H}_{kj}}

\begin{document}

\title{Submatrices with non-uniformly selected random supports and insights into sparse approximation}



\author{\name Simon Ruetz \email simon.ruetz@uibk.ac.at\\
	\name Karin Schnass \email karin.schnass@uibk.ac.at\\
	\addr  
	University of Innsbruck\\
	Technikerstra\ss e 13\\
	6020 Innsbruck, Austria}

\editor{}

\maketitle

\begin{abstract}
In this paper we derive tail bounds on the norms of random submatrices with non-uniformly distributed supports. We apply these results to sparse approximation and conduct an analysis of the average case performance of thresholding, Orthogonal Matching Pursuit and Basis Pursuit. As an application of these results we characterise sensing dictionaries to improve average performance in the non-uniform case and test their performance numerically.
\end{abstract}

\begin{keywords}
\noindent random submatrices, non-uniform sampling, matrix Chernoff, sparse approximation
\end{keywords}

\section{Introduction}\label{sec:intro}

\noindent
\textbf{Motivation:} 
In sparse approximation, the goal is to find a sparse solution to an underdetermined system of linear equations. A signal $y \in \R^d$ is assumed to be a linear combination of a small number $S \ll d$ of elements $\atom_i$, called atoms, out of a larger set, called the dictionary. Denoting the dictionary by $\dico = (\atom_1,\ldots,\atom_K)  \in \R^{d \times K}$ and by $\dico_I$ the restriction to the columns indexed by the set $I$, called the support, one assumes that
\[
y \approx \sum_{k \in I}\atom_k x_k = \dico_I x_I \quad \text{s.t.} \quad |I| = S.
\]
The sparse approximation problem amounts to finding the vector $x$ and its support $I$ given the dictionary $\dico$ and signal $y$. In general, this is a NP-hard optimisation problem, hence sparse approximation algorithms such as thresholding, Orthogonal Matching Pursuit (OMP) and Basis Pursuit (BP) were proposed. It turns out that in order to prove support recovery guarantees for these algorithms, information about the extreme singular values of $\dico_I$ is needed. \\
Let $\| \cdot \|_{2,2}$ denote the operator norm and $\mathbb{I}$ the identity matrix. Deterministic methods to bound $\| \dico_I\transp \dico_I - \mathbb{I} \|_{2,2}$ for \textit{arbitrary} supports $I$ are of limited use since the restrictions on the dictionary $\Phi$ are too stringent. This started the study of random collections of columns of the dictionary $\dico$. In~\cite{tr08} it was first shown that under rather mild conditions on the dictionary $\dico$, \textit{most} subdictionaries $\dico_I$ are close to an isometry - i.e. $\| \dico_I \transp \dico_I - \mathbb{I}\|_{2,2} \leq \vartheta_0 < 1$, with later improvements in~\cite{chda12}. So far, all available results on the conditioning of random subdictionaries rely on the supports $I$ to be drawn from the uniform distribution. Unfortunately this assumption is rarely satisfied for practically relevant signal classes, where some atoms of the underlying dictionary are usually more likely to appear in a sparse representation than others.\\ 
To demonstrate this non-homogeneity, we conduct the following small experiment. We take the 2D Haar-Wavelet decomposition of all normalised $64 \times 64$ patches from the image \textit{Peppers} and apply a threshold\footnote{The threshold is inspired by the expected size of the largest inner product of a wavelet with noise drawn uniformly at random from the unit sphere.} of $\sqrt{\log(d)/d}/6$ for $d=64^2$ to the coefficients to get sparse approximations. We then count how often each atom has a non-zero coefficient to get a proxy for its inclusion probability in a sparse support $I$. Figure~\ref{wav_coeff} shows the relative frequency of each element of the 2D Haar-Wavelet basis.
\begin{figure}[ht]
\begin{subfigure}[c]{0.32\linewidth}  
  \centering
  \includegraphics[width=\linewidth]{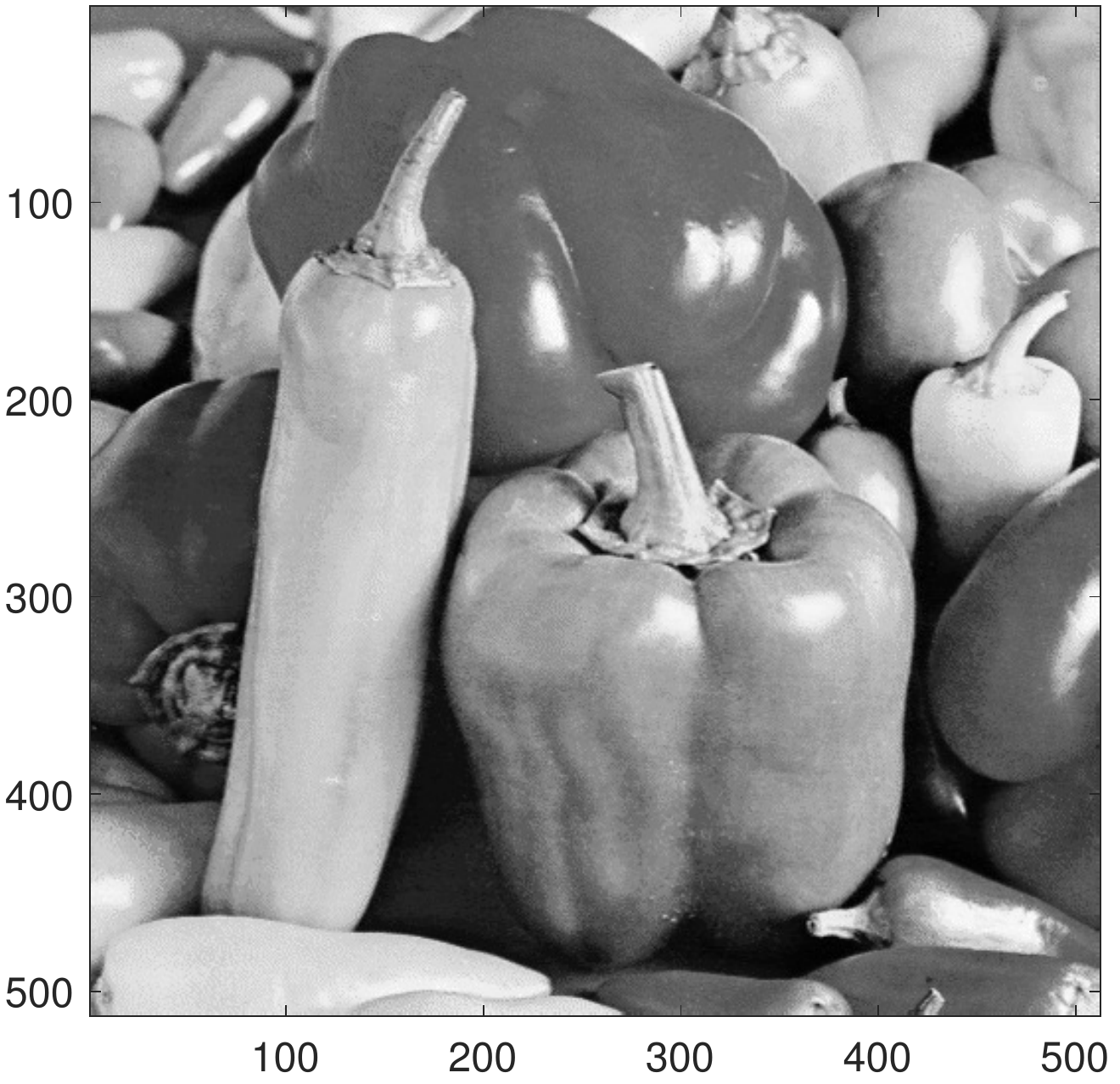}
  \caption{}
\end{subfigure}
\begin{subfigure}[c]{0.33\linewidth}  
  \centering
  \includegraphics[width=\linewidth]{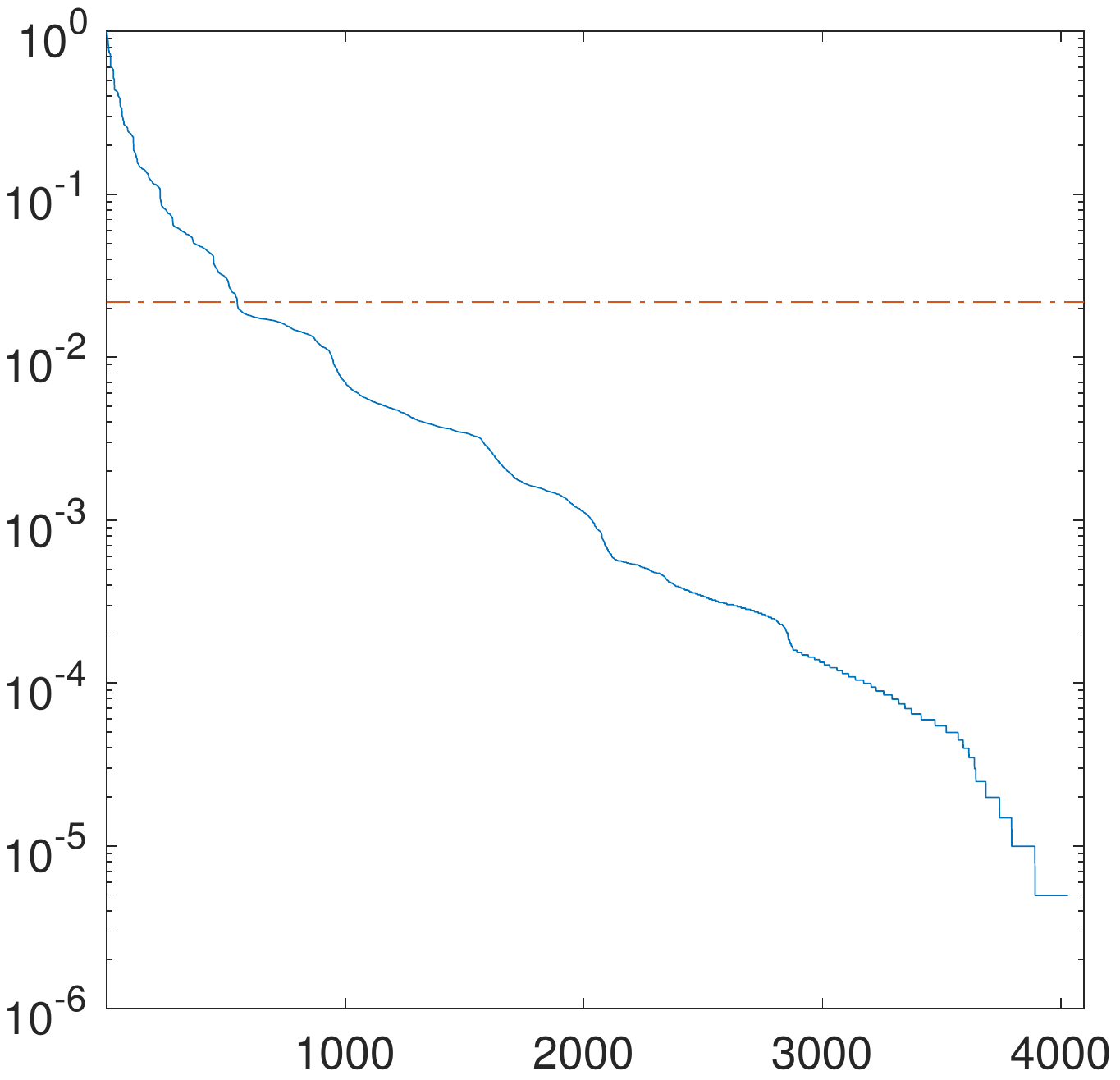}
  \caption{}
\end{subfigure}
\begin{subfigure}[c]{0.32\linewidth}  
  \centering
  \includegraphics[width=\linewidth]{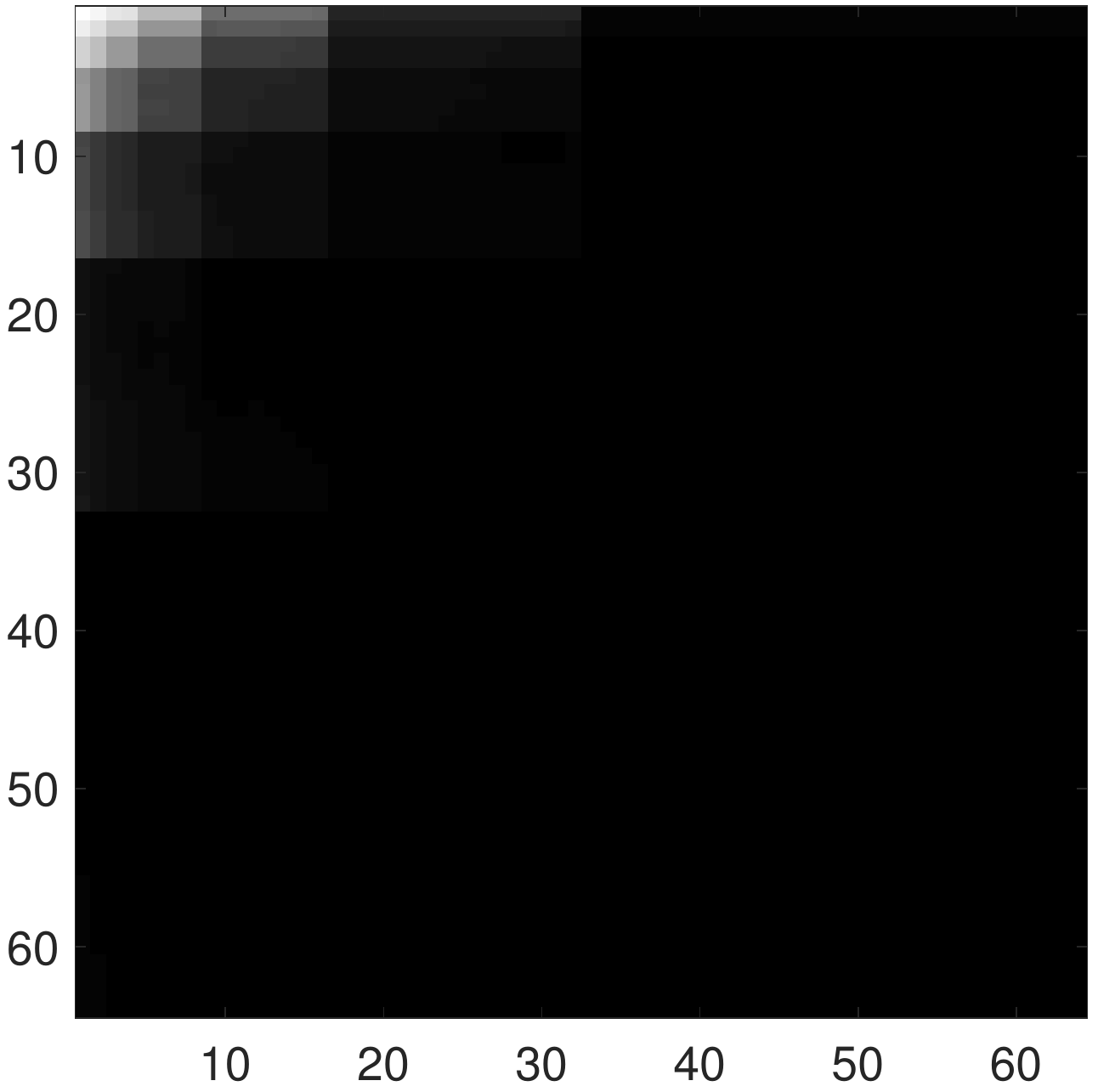}
  \caption{}
\end{subfigure}
\caption{(a) Original image from which the patches are extracted. (b) Relative frequency of wavelet coefficients above threshold (blue) - average frequency (red) on a log scale. (c) Locations of non-zeros coefficients in the 2D Haar-Wavelet basis - the higher the row or column index the smaller the corresponding wavelet}
\label{wav_coeff}
\end{figure}
It comes as no surprise that low frequency (large) wavelets are much more likely to appear in the sparse supports than high frequency (small) wavelets. So the supports of the sparse signals exhibit a non-uniform structure which previous results on the conditioning of random subdictionaries do not cover. We try to close this gap by defining two non-uniform support distributions and deriving tail bounds on the norms of the resulting random submatrices. This allows us to derive recovery guarantees of the sparse supports for a larger class of practically relevant signals.\\
\noindent\textbf{Prior work:} As mentioned above, Tropp~\cite{tr08} and Chr\'etien and Darses~\cite{chda12} derived concentration inequalities for the operator norm of random submatrices with uniformly distributed supports. These results were applied to BP showing that BP recovers the correct support and coefficients under rather mild conditions on the dictionary \cite{tropp08}. For OMP, similar results were developed in \cite{sc18omp}, whereas for thresholding average case results appeared in \cite{scva07}. \\
In \cite{pairs_dictionary_recovery} the dictionary $D$ is assumed to be a concatenation of two dictionaries $\phi$ and $\psi$, i.e. $D = ( \phi, \psi)$. There a concentration inequality on the extreme singular values of submatrices that consist of a \textit{fixed} set of columns with cardinality $n_a$ of the first dictionary and a \textit{random} set of columns $n_b$ of the second dictionary are derived. This allows to model signals where some atoms are known to be in the support while some others are picked uniformly at random.\\
The idea of using the structure of sparse signals to improve recovery of the sparse coefficients can also be found in the field of compressed sensing. The aim in compressed sensing is to recover a sparse signal $y\in \R^d$ from an incomplete set of linear measurements $z = Ay$, where $A\in \R^{m\times d}$ and $m\ll d$, \cite{carota06, do06cs}. The signal $y$ is assumed to be sparse or compressible in some (orthonormal) basis or frame $\dico$, i.e. $y = \dico x$ for a sparse coefficient vector $x$. \\
From a theoretical point of view the best measurement matrices $A$, achieving the smallest $m$ for a given sparsity level $S$, are random matrices. Unfortunately in many practical applications it is not possible or efficient to use random matrices, since they cannot be realised by the underlying physical measurement process, such as in compressed magnet resonance imaging (MRI). Instead one is given an (often orthonormal) measurement matrix $\pdico\in \R^{d\times d}$ and has to find a \textit{subsampling pattern} $\Omega \subseteq \{1,... ,K \}$ which selects $m$ rows of $\pdico$, so that for $A=P_{\Omega} \pdico$ the signal $y$ resp. the coefficients $x$ can be reliably reconstructed from $z = Ay = A \dico x = \bar A x $.\\  
As in sparse approximation, rather strong assumptions on the matrix $A \dico = \bar A$ are needed in order to guarantee recovery for all sparse $x$. In \cite{candes11} the elements of $\Omega$ were assumed to be chosen uniformly at random in order to employ probabilistic arguments to derive sufficient conditions for recovery for relatively small $m$. Over the years, various different subsampling strategies - most of them highly non-uniform - were proposed (see for example \cite{boyer13, weiss13_2, Vander11, adhaporo17, krwa12}). Underlying the success of these variable density sampling strategies is the highly non-uniform structure of the sparse supports. So it was shown that previous lower bounds on the size of $m$ are too pessimistic and performance can be improved if the subsampling pattern takes the support structure of the sparse signals into account \cite{adhaporo17, krwa12}.\\
\noindent\textbf{Contribution:} We derive tail bounds on the operator norm of non-uniformly chosen submatrices. The supports are assumed to follow either a Poisson sampling model or a rejective sampling model thus allowing us to model a large class of non-uniform distributions. Our results rely on a generalisation of a Theorem by Chr\'etien and Darses~\cite{chda12}. The main tool to handle non-uniformly distributed $S$-sparse supports is a kind of Poissonisation argument where we provide a generalised version of Lemma 4.1 of~\cite{hajek1964}. We apply these results to derive sufficient conditions for sparse approximation to work with high probability for thresholding, OMP and BP. In the CS setup this analysis provides a criterion to decide between two possible measurement matrices $A_1$ and $A_2$ depending on the frequency of the basis elements. Further, if there is no design freedom for the dictionary or CS matrix, we show how to incorporate this prior information about the coefficient distribution into the algorithms using the ideas of preconditioning and sensing dictionaries.\\
\noindent\textbf{Organisation:} Section \ref{sec:notations} collects our notations and defines the setting we work in. In Section \ref{sec:main} we state our results on norms of non-uniformly distributed random submatrices and apply those concentration inequalities to sparse approximation in Section \ref{sec:applications}. Finally we incorporate this knowledge in the construction of special sensing dictionaries in Section \ref{sec:sensing} and show how they improve performance.



\noindent
\section{Notation and setting}\label{sec:notations}
A quick note on the notation used throughout this text. Let $A \in \R^{d \times K}$ and $B \in \R^{K\times m}$. By $A_k$ and $A^k$ we denote the $k$-th column and $k$-th row of $A$ respectively and by $A\transp$ the transpose of the matrix $A$. For $1 \leq p,q,r \leq \infty$ we set $\| A \|_{p,q} := \max_{\| x\|_q = 1} \| Ax \|_p $. Recall that $\| A B \|_{p,q} \leq \|A \|_{q,r} \|B\|_{r,p}$ and $\|Ax \|_{q} \leq \|A\|_{q,p} \|x \|_{p}$. Frequently encountered quantities are
\[
\| A \|_{\infty, 2} = \max_{k \in \{1, \dots , d \}}\|A^k \|_2 \quad \text{and} \quad \| A \|_{2, 1} = \max_{k \in \{1, \dots , K \}}\|A_k \|_2,
\]
denoting the maximum $\ell_2$-norm of a row and the maximum $\ell_2$-norm of a column of $A$ respectively. Note that $\| A \|_{\infty,2} = \|A\transp \|_{2,1}$. Further note that $\| A \|_{\infty,1}$ simply is the maximum absolute entry of the matrix $A$. For ease of notation we sometimes write $\| A \| = \| A \|_{2,2}$ for the operator norm - corresponding to the largest absolute singular value of $A$. For a vector $v \in \R^{d}$, we denote by $\|v\|_{\min} := \min_{i} |v_i|$ the smallest absolute value of $v$ and $\|v\|_{\max} := \|v \|_{\infty}$ the maximal absolute value of $v$. For a subset $I\subseteq \mathbb{K}:=\{ 1,\dots ,K \}$, called the support, we denote by $A_I \in \R^{d \times S}$  the submatrix with columns indexed by $I$ and by $A_{I,I} \in \R^{S \times S}$ the submatrix with columns and rows indexed by $I$. We denote by $A_I^{\dagger}$ the Moore-Penrose pseudo inverse of the matrix $A_I$ and by $P(A_J) := A_J A_J^{\dagger}$ the projection onto the column span of $A_J$. As was noted in the introduction we want the supports to follow a non-uniform distribution, allowing some columns, called atoms, to be picked more frequently than others. We are going to use the following two sampling models which define two probability measures on $\mathcal{P}(\mathbb{K})$ that allow us to model non-uniform distributions for our supports.
\begin{definition}[Poisson sampling]
Let $\delta_j$ denote a sequence of $K$ independent Bernoulli $0$-$1$ random variables with expectation $p_j$ such that $\sum_{j = 1}^K p_j = S$. We say the supports $I$ follow the Poisson sampling model, if
\[
I  := \left\{ i \; \middle| \; \delta_i  = 1\right\}.
\]
Each support $I \subseteq \mathbb{K}$ is chosen with probability
\begin{equation}\label{ind_dist}
    \P(I)=\prod_{i \in I}p_i\prod_{j \notin I}(1-p_j).
\end{equation}
\end{definition}
Supports following a Poisson sampling model have (by definition of the Bernoulli r.v.) cardinality $S$ \textit{on average}. This comes with the big advantage that the probability of one atom appearing in the support is independent of the others, allowing us to make use of concentration inequalities for sums of independent random matrices later on. The drawback of this model is that the supports are not exactly $S$ sparse. This can be achieved by keeping only those supports that have cardinality $S$ and \textit{throwing away} the rest. This amounts to simply conditioning the above Poisson sampling model on the event that exactly $S$ of the Bernoulli r.v. are equal to $1$, leading to our second support distribution model.
\begin{definition}[Rejective sampling]
Let $\delta_j$ denote a sequence of $K$ independent Bernoulli $0$-$1$ random variables with expectation $p_j$ such that $\sum_{j =1}^K p_j = S$ and denote by $\P$ the probability measure of the corresponding Poisson sampling model. We say our supports follow the rejective sampling model, if each support $I \subseteq \mathbb{K}$ is chosen with probability
\begin{equation}\label{cond_dist}
    \P_S(I) := \P(I \; | \;  |I| = S) = \begin{cases}
               c \prod_{i \in I}p_i\prod_{j \notin I}(1-p_j) \quad& \mbox{if} \quad |I|=S\\
               0 \quad& \mbox{else}
            \end{cases},
\end{equation}
where $c$ is a constant to ensure that $\P_S$ is a probability measure.
\end{definition}
The distributions of the supports in the above two sampling models are uniquely defined by the expectations of the Bernoulli random variables. For more information on Poisson and rejective sampling, we refer the interested reader to~\cite{hajek1964}. We call the square diagonal matrix $\weights := \diag((\sqrt{p_k})_k)$ the weight matrix. Let $\proj$ be the square diagonal \textit{selector matrix} whose diagonal entries are the $\delta_j$, i.e. $\proj  = \diag((\delta_k )_k)$ and denote by $\proj'$ an independent copy of $\proj$. Further let $\vec{A}_{jk} = A_{jk} e_j \otimes e_k$ be the matrix with only non-zero entry $A_{jk}$. This allows us to write
\begin{equation*}
    \proj A \proj =\sum_{i,j} \delta_i \delta_j \vec{A}_{ij}.
\end{equation*}
Note that by properties of the operator norm, the two random variables $\| A_I \|$ and $\| A \proj \|$ have the same distribution. 
\section{Main results}\label{sec:main}
We now present our main results on submatrices whose support is sampled from a non-uniform distribution. We begin by stating the concentration inequality for the operator norm of non-uniformly picked random submatrices, before turning to some special cases arising in sparse approximation. Then we state a concentration inequality for the maximal row norm of random column-submatrices. Lastly we state and proof a kind of Poissonisation argument - of independent interest - which is key for our proofs. Note that we state our results only for the rejective sampling model, but they hold for the Poisson sampling model as well - see Remark~\ref{rem:1}.

\subsection{Operator norm of random submatrices}
The aim is to get a tail bound for the random variable $\| H_{I,I}\|_{2,2}$, where $I$ is distributed according to the models introduced above and $H$ is a matrix with zero diagonal. As expected, the result shows how the more frequently picked entries have a higher impact on the operator norm than less important ones.
\begin{theorem} \label{them:1}
Let $H \in \R^{K \times K}$ be a matrix with zero diagonal and assume $I \subseteq\mathbb{K}$ is chosen according to the rejective sampling model with probabilities $p_1, \dots , p_K$ such that $\sum_{i = 1}^K p_i = S$. Further let $W$ denote the corresponding weight matrix. Then, for all $r \geq 2 e^2 \| \weights H \weights \|_{2,2}$
\begin{equation*}
    \P_S \bigg( \left\| H_{I,I} \right\|_{2,2} \geq r \bigg) \leq 216 K \exp{\left(-\min\left\{\frac{r^2}{4e^2\|  H \weights \|_{\infty,2}^2 },\frac{r^2}{4e^2\|\weights H \|_{2,1}^2},\frac{r}{2\| H \|_{\infty,1}} \right\}\right)}.
\end{equation*}
\end{theorem}

\begin{proof}[Outline]
We follow the proof that appeared in Chr\'etien and Darses~\cite{chda12} with some minor changes to account for the non-uniformly distributed supports and the extension to non-symmetric matrices. Their proof consists of roughly three steps. First they bound the failure probability of the rejective sampling model by the independent Poisson sampling model 
\[
\P_S \left ( \| \proj H \proj \|_{2,2} \geq r  \right) \leq 2 \P\left ( \| \proj H\proj \|_{2,2} \geq r  \right).
\]
Then they use a decoupling argument to make the selection of rows and columns independent, i.e.
\[
\P\left ( \| \proj H\proj \|_{2,2} \geq r  \right) \leq 72 \P \left ( \| \proj H\proj' \|_{2,2} \geq r/2  \right),
\]
where $\proj'$ is an independent copy of $\proj$.
Then they apply the matrix Chernoff inequality three times to finish the proof. Our proof in the non-uniform, non-symmetric case follows the above outline very closely. The main difficulty lies in bounding the rejective model by the Poisson model, which is why we had to provide Lemma~\ref{lemma:f}. The second and third steps are straightforward extensions of their argument. For the sake of completeness we provide a detailed proof in the appendix. 
\end{proof}
\subsubsection{Special cases - hollow (cross)-Gram matrices}
In this subsection we look at the special case $H = \dico \transp \dico - \mathbb{I}$ that appears naturally in the sparse approximation framework. Previous results showed that success of recovery depends on the coherence $\mu := \max_{i \neq j}|\ip{\phi_i}{\phi_j}|$ and the conditioning of the subdictionary $\dico_I$, i.e.
\begin{equation*}
    \vartheta_I := \| \dico_I\transp \dico_I - \mathbb{I} \|_{2,2} = \max\left\{ \lambda^2_{\max}(\dico_I) -1  , 1 - \lambda^2_{\min}(\dico_I)\right\}.
\end{equation*}
Here $\lambda_{\max}^2$ and $\lambda_{\min}^2$ denote the biggest and smallest eigenvalue of $\dico_I \transp \dico_I$ respectively. In this setting, the matrix $H := \dico \transp \dico - \mathbb{I}$ is called the hollow Gram matrix and we call $\mu: = \max_{i \neq j} |\ip{\atom_i}{\atom_j}| = \| H \|_{\infty, 1}$ the coherence. Applying Theorem \ref{them:1} to this matrix, we get the following bound on $\vartheta_I$.
\begin{corollary}\label{cor:1}
Let $\dico \in \R^{d \times K}$ be a dictionary with unit norm columns and assume $I \subseteq\mathbb{K}$ is chosen according to the rejective sampling model with probabilities $p_1, \dots , p_K$ such that $\sum_{i = 1}^K p_i = S$. Further let $W$ denote the corresponding weight matrix. Then, for all $r \geq 2 e^2 \| \weights H \weights \|_{2,2}$
\begin{equation*}
    \P_S \bigg( \left\| \dico_I\transp \dico_I - \mathbb{I}\right\|_{2,2} \geq r \bigg) \leq 216 K \exp{\left(-\min\left\{\frac{r^2}{4e^2\|  H \weights \|_{\infty,2}^2 }, \frac{r}{2\mu} \right\}\right)}.
\end{equation*}
\end{corollary} \label{cor:4}
In this setting $H$ is symmetric, hence $H\transp \weights = H \weights$. The result can be used to bound
\begin{equation*}
    \P_S\left(\| \dico_I \|_{2,2} \gtrless \sqrt{1 \pm r}\right) \quad \text{and} \quad \P_S\left(\| (\dico_I \transp \dico_I)^{-1} \|_{2,2} \geq \frac{1}{1-r}\right).
\end{equation*}
This comes in handy when trying to prove recovery guarantees of sparse approximation algorithms later in this text.\\
Another frequently arising quantity is the cross-Gram matrix $H := \pdico \transp \dico - \diag(\pdico\transp \dico)$, where $\dico$ and $\pdico$ are dictionaries. In this setting, we call $\hat{\mu}: = \max_{i \neq j} |\ip{\atom_i}{\patom_j}|$ the cross-coherence. Applying Theorem \ref{them:1} yields
\begin{corollary}\label{cor:2}
Let $\pdico, \dico \in \R^{d \times K}$ be dictionaries and assume $I \subseteq\mathbb{K}$ is chosen according to the rejective sampling model with probabilities $p_1, \dots , p_K$ such that $\sum_{i = 1}^K p_i = S$. Further let $W$ denote the corresponding weight matrix. Then, for all $r \geq 2 e^2 \| \weights H \weights \|_{2,2}$
\begin{equation*}
    \P_S \bigg( \! \left\| \pdico_I\transp \dico_I - \diag(\pdico_I\transp \dico_I)\right\| \geq r \! \bigg) \leq 216 K \exp{\!\left(-\min\!\left\{\frac{r^2}{4e^2\|  H \weights \|_{\infty,2}^2 },\frac{r^2}{4e^2\|\weights H \|_{2,1}^2 }, \frac{r}{ 2\hat{\mu}} \right\}\right)}.
\end{equation*}
\end{corollary} 
Note that in contrast to Corollary~\ref{cor:4} the matrix $H$ is not symmetric any more, hence we need to control both $\|H\weights \|_{\infty,2}$ and $\|\weights H \|_{2,1}$. 
In contrast to previous works the above results are in terms of the maximal row norm of the weighted Gram matrix. By using the bounds 
\begin{align*} 
&\|  H \weights \|_{\infty,2} \leq \|  \pdico \transp \dico \weights  \|_{\infty,2} \leq \|\pdico \transp \|_{\infty,2} \| \dico  \weights \|_{2,2} = \| \dico  \weights \|_{2,2},\\
&\|\weights H \|_{2,1} = \|   H \transp  \weights \|_{\infty,2} \leq \| \dico \transp \|_{\infty,2}  \| \pdico  \weights \|_{2,2} = \| \pdico  \weights \|_{2,2},\\
&\| \weights H \weights \|_{2,2}  \leq \|  \pdico \weights\|_{2,2} \| \dico \weights \|_{2,2}
\end{align*}
one would get bounds similar in spirit to the results of Chr\'etien and Darses~\cite{chda12} and Tropp~\cite{tr08}.\\ 
We stick to the quantities $\|  H \weights \|^2_{\infty,2}$ and $\|  \weights H  \|^2_{2,1}$ to see how the weights of the distribution interact with the structure of $H$. Intuitively the above results state that the more frequently an atom is picked, the less coherent it should be to all the other atoms in order for a random submatrix to be well-conditioned. \\
The generality of this result allows for $p_i \in [0,1]$, which thus includes models where some atoms are already known to be in the support and some to not appear at all. This allows for models where a dictionary $D$ is a concatenation of two dictionaries $\phi$ and $\psi$, i.e. $D = (\phi, \psi)$ and the submatrix of interest consists of a \textit{fixed} set of columns with cardinality $n_a$ of the first dictionary and a \textit{random} set of columns $n_b$ of the second dictionary. Such a scenario can easily be modeled by setting the $p_i$ and the weight matrix $\weights$ accordingly and would yield similar results to \cite{pairs_dictionary_recovery}.
\subsection{Maximum row norm of a random restriction}
Another frequently encountered random variable in sparse approximation is the maximal row norm $\| H_I\|_{\infty,2}$. Given a weight matrix $\weights$, the following Lemma states that one can expect this quantity to be approximately of size $\| H \weights \|_{\infty,2}$. This can be significantly smaller than the worst case $\max_{i,j}|H_{i,j}| \sqrt{S} $ for $|I| \leq S$, depending on the structure of $H$ and $\weights$. Plugging in $H = \pdico \transp \dico - \diag(\pdico \transp \dico)$ we again see that the more frequently picked atoms should have smaller coherences in order for $\| H \weights \|_{\infty,2}$ to be small. This result is an integral part of the proof of Theorem~\ref{them:1} and hence we defer its proof to the appendix.
\begin{lemma}\label{lem:1}
Let $H \in \R^{d \times K}$ be some matrix. Assume $I \subseteq \mathbb{K}$ is chosen according to the rejective sampling model with probabilities $p_1, \dots , p_K$ such that $\sum_{i = 1}^K p_i = S$. Further let $W$ denote the corresponding weight matrix. Then, for all $v >0$
\begin{equation*}
    \P_S \left( \| H _I \|_{\infty,2} \geq v \right) \leq 2K \left( e \frac{\| H \weights \|_{\infty,2}^2}{v^2}  \right) ^{\frac{v^2}{\mu^2}}.
\end{equation*}
\end{lemma}
\subsection{Poissonisation argument in the non-uniform case}
As already mentioned, we have to bound the failure probability under the rejective sampling model by the failure probability under the Poisson sampling model in order to apply concentration inequalities for sums of independent random variables. In the uniform case the following lemma is not needed, as one can argue that the supports can also be sampled by drawing one atom after the other to get a uniform support distribution - see Claim (3.29) p. 2173 in \cite{candes2009}. For the non-uniform case it is not that easy. Lemma 4.1 of~\cite{hajek1964} almost provides the result that we need, but has too restrictive assumptions on the expectations $p_i$. Therefore we prove\ifthenelse{\boolean{arxiv}}{\footnote{The result might be known but extremely well hidden, thus forcing us to prove it.}}{} the following result which does not have any constraints on the expectations $p_i$. 
\begin{lemma}[Poissonisation]\label{lemma:f}
Denote by $\P$ the probability measure corresponding to the Poisson sampling model \eqref{ind_dist} and by $\P_S$ the probability measure corresponding to the rejective sampling model \eqref{cond_dist} - both with the same weight matrix $\weights$. Let $f: \mathcal{P}(\mathbb{K}) \mapsto \{0,1 \}$ be such that for all $I,J \in \mathcal{P}(\mathbb{K})$
\begin{equation*}
  f(I) \leq f(J) \quad \text{if} \quad I \subseteq J.  
\end{equation*}
Then for all $I \subseteq \mathbb{K}$
\begin{align*}
    \P_S \left( f(I) = 1 \right) \leq 2 \; \P \left( f(I) = 1 \right).
\end{align*}
\end{lemma}
\begin{proof}
Note that the conditions on $f$ imply that if $f(J) =0$ for some $J$, then $f(I) = 0$ for all $I \subset J$. We start by showing that for $0 \leq T \leq K-1$ we have
\begin{equation*}
    \P\big(f(I) = 1 \; \big| \; |I| = T \big) \leq \P\big(f(I) = 1 \; \big| \; |I| = T +1 \big).
\end{equation*}
Expanding the conditional probability we get
\begin{equation*}
     \frac{\sum_{I: |I|=T} f(I)\P(I)}{\sum_{I: |I|=T}\P(I)} \leq \frac{\sum_{J: |J|=T+1} f(J) \P(J)}{\sum_{J: |J|=T+1}\P(J)},
\end{equation*}
which is equivalent to 
\begin{align}\label{lemma:f:1}
   \sum_{I: |I|=T} f(I)\P(I) \sum_{J: |J|=T+1}\P(J) \leq \sum_{J: |J|=T+1} f(J)\P(J) \sum_{I: |I|=T}\P(I).
\end{align}
By combining the sums on both sides and subtracting 
\begin{equation*}
    \sum_{J: |J|=T+1 }\sum_{I: |I| = T}\P(J)\P(I)f(I)f(J)
\end{equation*}
on both sides we see that \eqref{lemma:f:1} is equivalent to
\begin{equation} \label{lemma:f:2}
    \sum_{J: |J|=T+1 }\sum_{I: |I| = T}\P(J)\P(I)f(I)[1-f(J)] \leq \sum_{J: |J|=T+1} \sum_{I: |I| = T }\P(J)\P(I) f(J)[1-f(I)]
\end{equation}
Now the crucial step is to see that we can partition these sums in a very special way. For a pair $(I,J)$, by definition of the Poisson sampling model, we can write $\P(I)\P(J)$ in the following way
\begin{equation*}
    \P(I)\P(J) = \prod_{i \in I}p_i\prod_{j \notin I}(1-p_j)\prod_{i \in J}p_i\prod_{j \notin J}(1-p_j) = \prod_{i \in I \cap J}p_i^2\prod_{i \in I \triangle J}p_i(1-p_i)\prod_{j \notin I \cup J}(1-p_j)^2.
\end{equation*}
This implies that for two pairs $(I,J)$, $(I',J')$ with 
\[
I \cap J = I' \cap J' \quad \text{and} \quad  I \triangle J = I' \triangle J' \quad \text{we have} \quad  \P(I)\P(J) = \P(I')\P(J').
\]
This allows us to define natural partitions on the set of pairs $(I,J)$ such that the probability $\P(I)\P(J)$ is constant on each partition:
Let $k \in \mathbb{T}$, $A\subseteq \mathbb{K}$ with $|A| = k$ and $B \subseteq \mathbb{K}\setminus A $ with $|B| = 2(T-k)+1$. $A$ will be the intersection and $B$ will model the symmetric difference of the sets $I$ and $J$ respectively. For such a combination of $A,B$ we define
\begin{equation*}
    \mathcal{Q}_{A,B} := \left\{(I,J) : I,J \subseteq \mathbb{K} , |I|=T,|J|= T+1,I \cap J =A, I\triangle J = B\right\}.
\end{equation*}
Note that each pair $(I,J)$ with $|I|=T$, $|J|= T+1$ can be \textit{uniquely} assigned to one $\mathcal{Q}_{A,B}$. So if
\begin{equation}\label{lemma:f:3}
    \sum_{(I,J) \in \mathcal{Q}_{A,B}} f(I)[1-f(J)] \leq \sum_{(I,J) \in \mathcal{Q}_{A,B}} f(J)[1-f(I)]
\end{equation}
for all possible choices of $A,B$ then \eqref{lemma:f:2} follows and we are done.\\
We start with the special case $|A| = 0$ and fix $B \subseteq \mathbb{K}$ with $|B| = 2T+1$. With slight abuse of notation we write $I^c := B\setminus I $ for the complement in $B$. With this notation \eqref{lemma:f:3} becomes
\begin{align*}
     \sum_{\substack{I \subseteq B \\ |I| = T}}f(I)(1-f(I^c))  \leq  \sum_{\substack{J \subseteq B \\ |J| = T+1}} f(J)(1-f(J^c)).
\end{align*}
Remembering that $f(I) \leq f(I \cup \{i\})$ and $f(J) \geq f(J\setminus \{i\})$ we get
\begin{align*}
    \sum_{\substack{I \subseteq B \\ |I| = T}}f(I)(1-f(I^c)) &= \sum_{\substack{I \subseteq B \\ |I| = T}}f(I)(1-f(I^c)) \frac{1}{T+1}\sum_{i \in I^c}1 \\
    & =  \frac{1}{T+1}  \sum_{\substack{I \subseteq B \\ |I| = T}}f(I)(1-f(I^c)) \sum_{i \in I^c}f(I\cup \{i\})(1-f(I^c \setminus  \{i\})) \\
     & \leq   \frac{1}{T+1}  \sum_{\substack{I \subseteq B \\ |I| = T}} \sum_{i \in I^c}f(I\cup \{i\})(1-f(I^c \setminus  \{i\})) \\
     & =  \frac{1}{T+1}  (T+1) \sum_{\substack{J \subseteq B \\ |J| = T+1}} f(J)(1-f(J^c)).
\end{align*}
If $|A| > 0$ then the same argument as above replacing $f(\cdot )$ with $f(A \cup \cdot )$ and $T$ with $T-|A|$ yields \eqref{lemma:f:3} for all possible choices of $A$ and $B$. Thus we get
\[
\P\big(f(I) = 1 \; \big| \; |I| = T \big) \leq \P\big(f(I) = 1 \; \big| \; |I| = T +1 \big).
\]
Now we are finally in a position to prove our result. Note that
\begin{align*}
    \P \left( f(I) = 1 \right) &= \sum_{k=1}^K \P \left( f(I) = 1 \; \middle| \;  |I|=k\right)\P \left( |I| = k \right)\\
    &\geq  \P \left(f(I) = 1 \; \middle| \;  |I|=S\right)\sum_{k=S}^K \P \left( |I| = k \right)\\
    &\geq \P_S \left( f(I) = 1 \right) \cdot \frac{1}{2} ,
\end{align*}
where the last inequality follows from Theorem 3.2 of~\cite{jogdeo1968} which says that if the mean number of successes of $K$ independent trials is an integer $S$, the median is also $S$. 
\end{proof}
\begin{remark}\label{rem:1}
Applying the above result on the functions $f_1(I): = \mathbbm{1}_{\left\{\| H_{I,I} \|_{2,2} \geq t \right\}}$ and $f_2(I): = \mathbbm{1}_{\left\{\| H_I \|_{\infty,2} \geq t \right\}}$ we get 
\[
\P_S \left (\| H_{I,I} \|_{2,2} \geq r  \right) \leq 2 \P\left ( \| H_{I,I} \|_{2,2} \geq r  \right)
\]
and
\[
\P_S \left( \| H _I \|_{\infty,2} \geq v \right)  \leq 2 \P \left( \| H _I \|_{\infty,2} \geq v \right) .
\]
Even though we stated our results only for the rejective sampling model, all of our proofs consist of first bounding the failure probability under the rejective sampling model by the failure probability under Poisson sampling model. Hence all of our results hold for the Poisson sampling model as well, with the failure bound actually improved by a factor $1/2$.
\end{remark}

\section{Application to sparse approximation}\label{sec:applications}
In this section we apply the derived result to sparse approximation. The starting point of sparse approximation is an underdetermined system of linear equations for which one tries to find the sparsest solution. Assuming that the signal $y$ is a linear combination of $S$ columns of a dictionary $\dico$, we show under which conditions sparse approximation algorithms are successful. To that end we define the following statistical model of our signals.
\begin{definition}[Signal model]\label{signal_model}
We model our signals as
\begin{equation*}
    y = \dico_I x_I =  \sum_{k = 1}^S \atom_{i_k} x_{i_k}   , \quad x_{i_k} = c_k \sigma_k, \; \; \forall k \in \{1, \dots , S \},
\end{equation*}
where $\dico \in \R^{d \times K}$ is a dictionary of $K$ normalised atoms, $I = \{i_1, \dots i_S \}$ is the random support and $c = \{ c_1, \dots c_S \}$ is an arbitrary sequence of strictly positive coefficients. We assume $I \subseteq\mathbb{K}$ is chosen according to the rejective sampling model with probabilities $p_1, \dots , p_K$ such that $\sum_{i = 1}^K p_i = S$ and denote by $W$ the corresponding weight matrix. Further we assume that the signs $\sigma_i$ form an independent Rademacher sequence, i.e. $\sigma_i = \pm 1$ with equal probability.
\end{definition}
This definition allows us to use probabilistic arguments to show that in the majority of cases, sparse approximation algorithms are able to recover the support under mild conditions on the dictionary $\dico$ and on the coefficients $x$. We denote by $\P_{y} : = \P_{\sigma, S}$ the product measure of the signs and the support and by $\mu := \max_{i \neq j} |\ip{\atom_i}{\atom_j}|$ the coherence of the dictionary $\dico$.
\subsection{Thresholding}
We start by considering the fastest and conceptually easiest sparse approximation algorithm. Thresholding works by finding the indices corresponding to the $S$ largest values of $|\ip{y}{\atom_i}|$, i.e.
\begin{align*}
    &\text{find} \quad J = \argmax_{|I| = S} \|\dico_I\transp y\|_1 \quad \text{and}\\
    &\text{reconstruct} \quad x_J = P(\dico_{J})y.
\end{align*}
In slight abuse of notation, let $\| c\|_{\min}:= \min_{i} |c_i|$. In~\cite{scva07}, average case results for thresholding were derived for the uniform case. There, a sufficient condition for thresholding to work with high probability was $S \mu^2 \log(K) \lesssim \| c \|_{\min}^2/ \| c \|_{\max}^2 $. We extend these results to the non-uniform case and show how the structure of the dictionary interacts with the distribution of coefficients.
\begin{theorem}[Thresholding]\label{them:thresh}
Assume that the signals follow the model in \eqref{signal_model}, where the support $I \subseteq\mathbb{K}$ is chosen according to the rejective sampling model with probabilities $p_1, \dots , p_K$ such that $\sum_{i = 1}^K p_i = S$. Further let $W$ denote the corresponding weight matrix and denote by $H = \dico\transp \dico - \mathbb{I}$ the hollow Gram-matrix. If
\begin{align*}
    \mu^2 \leq \frac{\| c \|_{\min}^2}{8 \| c \|_{\max}^2 \log( 4K/\varepsilon)}, \quad \quad \text{and} \quad \quad
    \| H \weights \|^2_{\infty,2} \leq \frac{\| c \|_{\min}^2}{8 e^2 \| c \|_{\max}^2\log(4K/\varepsilon)},
\end{align*}
then thresholding recovers the support with probability at least $1- \varepsilon$.
\end{theorem}
\begin{proof}
By definition of the algorithm, thresholding recovers the full support if
\[
\| \dico_{I^c}\transp y \|_{\max} <  \| \dico_I\transp y\|_{\min}.
\]
Note that the signals have two sources of randomness, $\sigma$ and $I$. Plugging in the definition of $y$ we derive a bound on the failure probability
\begin{align*}
\P_y(\| \dico_I\transp y\|_{\min} <  \| \dico_{I^c}\transp y \|_{\infty}) & =  \P_y \left( \| \dico_I\transp \dico_I x_I \|_{\min} <  \| \dico_{I^c}\transp \dico_I x_I \|_{\infty} \right) \\
&\leq  \P_y \left( \|c \|_{\min} - \| (\dico_I\transp \dico_I - \mathbb{I}) x_I \|_{\infty} <  \| \dico_{I^c}\transp \dico_I x_I \|_{\infty} \right) \\
& \leq \P_y \left( \| c \|_{\min} < 2 \| H_I x_I \|_{\infty}\right).
\end{align*}
Where we used that $x_{i_k} = \sigma_k c_k $, where $\sigma \in \R^{S}$ is an independent Rademacher sequence. Now as the signs $\sigma$ are independent from the support $I$, we can apply Hoeffding's inequality to each entry of $H_I\sigma$ (Lemma~\ref{hoeff}) and use Lemma~\ref{lem:1} to get
\begin{align*}
\P_y(\| \dico_I\transp y\|_{\min} <  \| \dico_{I^x}\transp y \|_{\max}) \leq  \P_y\left( \| H_I x_I \|_{\infty} \geq \frac{\| c\|_{\min}}{2} \; \middle| \;  \|H_I\|_{\infty,2}  < \gamma\! \right) + \P_S \bigg( \|H_I\|_{\infty,2}  \geq \gamma  \bigg) \\ \leq 2K\exp{\left(- \frac{\| c\|_{\min}^2}{8 \| c\|_{\max}^2 \gamma^2} \right) } + 2K \left( e \frac{\| H \weights \|_{\infty,2}^2}{\gamma^2}  \right) ^{\frac{\gamma^2}{\mu^2}}.
\end{align*}
Setting $\gamma^2 = \frac{\| c\|_{\min}^2}{8 \| c\|_{\max}^2 \log(4K/\varepsilon)}$ we see that the conditions of the Theorem imply that the failure probability does not exceed $\varepsilon$.
\end{proof}
\subsection{OMP}
One of the most popular sparse approximation algorithms is the Orthogonal Matching Pursuit (OMP). This greedy algorithm finds the support iteratively, adding one index at a time to the current support. In every step, it picks the index of the atom which has the largest absolute inner product with the residual and then updates the residual. Initialising $r_0 = y$ and $J_0 = \emptyset$, it 
\begin{align*}
    &\text{finds} \quad j = \argmax_k |\ip{\atom_k}{r_i}| \quad \text{and}\\
    &\text{updates} \quad J_{i+1} = J_i \cup \{j\} \quad \text{resp.} \quad r_{J_{i+1}} = y - P(\dico_{J_{i+1}})y,
\end{align*}
until a stopping criterion is met. Hence to prove that OMP recovers the correct support, one needs to ensure that it picks an atom from the support in each step. So assume OMP has successfully found $J \subseteq I $ in the $i$-th step, it will find another correct atom if
\[
\| \dico_{I^c}\transp r_J \|_{\infty} < \| \dico_{L}\transp r_J \|_{\infty},
\]
where $L := I \setminus J$. Based on this observation we prove the following Theorem.
\begin{theorem}[OMP]\label{them:omp}
Assume that the signals follow the model in \eqref{signal_model}, where the support $I \subseteq\mathbb{K}$ is chosen according to the rejective sampling model with probabilities $p_1, \dots , p_K$ such that $\sum_{i = 1}^K p_i = S$. Further let $W$ denote the corresponding weight matrix. Assume that the hollow Gram-matrix $H = \dico\transp \dico - \mathbb{I}$ satisfies $\| \weights H \weights\|_{2,2} \leq \frac{1}{4 e^2}$. If
\begin{align*}
\| H \weights \|_{\infty,2}^2 &\leq  \min\left\{ \min_{L \subseteq \{ 1, \dots, S \} }\frac{\|c_L \|^2_{\infty}}{16 e^2\|c_L \|_2^2}, \frac{1}{16e^2\log(216 K/\varepsilon)} \right\}\quad \text{and} \\ \mu &\leq \min\left\{\min_{L \subseteq \{ 1, \dots, S \} }\frac{\|c_L \|_{\infty}}{4\|c_L \|_2 \sqrt{\log(218 K/\varepsilon)}} , \frac{1}{4\log(218K/\varepsilon)}\right\},
\end{align*}
then OMP recovers the correct support with probability at least $1 - \varepsilon$.
\end{theorem}
\begin{proof}
Set $\|\dico_{I}\transp \dico_{I} - \mathbb{I} \|_{2,2}\ =: \vartheta_I$ and assume that $\vartheta_I < 1/2$. We start by expanding the residual in step $i$ 
\[
r_J = y - P(\dico_{J})y = \dico_I x_I -P(\dico_{J}) \dico_I x_I = \dico_{I \setminus J} x_{I \setminus J} -\dico_{J} (\dico_{J}\transp\dico_{J})^{-1} \dico_{J}\transp \dico_{I \setminus J} x_{I \setminus J}
\]
Set $L := I \setminus J$. By definition, OMP finds another correct atom in the next step if
\begin{equation}\label{omp:eq:1}
\| \dico_{I^c }\transp( \dico_{L} x_{L} -\dico_{J} (\dico_{J}\transp\dico_{J})^{-1} \dico_{J}\transp \dico_{L} x_{L})  \|_{\infty} < \| \dico_{L}\transp (\dico_{L} x_{L} -\dico_{J} (\dico_{J}\transp\dico_{J})^{-1} \dico_{J}\transp \dico_{L} x_{L} ) \|_{\infty},
\end{equation}
i.e. the inner products with the residual of the remaining atoms in the support are bigger than the inner products with the residual of atoms outside the support. Writing this differently, we get the sufficient condition
\begin{align*}
\| \dico_{I^c }\transp \dico_{L}x_{L}  \|_{\infty} &+ \|\dico_{I^c }\transp  \dico_{J} (\dico_{J}\transp\dico_{J})^{-1} \dico_{J}\transp \dico_{L}x_{L}  \|_{\infty} \\ & < \| x_{L} \|_{\infty} - \| (\dico_{L}\transp \dico_{L} -\mathbb{I})x_{L}  \|_{\infty} - \| \dico_{L} \transp \dico_{J} (\dico_{J}\transp\dico_{J})^{-1} \dico_{J}\transp \dico_{L} x_{L} \|_{\infty},
\end{align*}
Note that
\begin{align*}
\max \left\{
\| \dico_{I^c}\transp \dico_{L}  \|_{\infty,2},
\| \dico_{I^c}\transp \dico_{J}  \|_{\infty,2},
\| \dico_{L}\transp \dico_{L} - \mathbb{I} \|_{\infty,2},
\| \dico_{L}\transp\dico_{J}\|_{\infty,2}  \right\}
\leq \|  H_I \|_{\infty,2}.
\end{align*}
So OMP works if
\begin{equation}\label{omp:eq:2}
2\|  H_I \|_{\infty,2} \|x_{L} \|_2 +2\|  H_I \|_{\infty,2} \|(\dico_{J}\transp\dico_{J})^{-1} \|_{2,2} \|\dico_{J}\transp \dico_{L}\|_{2,2}\| x_{L}  \|_{2} < \| x_{L} \|_{\infty},
\end{equation}
By properties of the operator norm we have $\| \dico_{J}\transp \dico_{L}\|_{2,2} \leq \vartheta_I $ and $\| (\dico_{J}\transp\dico_{J})^{-1}\|_{2,2} \leq \frac{1}{1-\vartheta_I}$. Plugging this into \eqref{omp:eq:2} we see that OMP will pick a correct atom in the next step, if
\[
\|  H_I \|_{\infty,2} \left(2+2\frac{\vartheta_I}{1-\vartheta_I}\right) < \frac{\| x_L \|_{\infty}}{\| x_L\|_2}.
\]
So on the set $ \left\{ \vartheta_I  < 1/2 \right\}$ the columns of $\dico_I$ are linearly independent and we need to have $\displaystyle \| H_I \|_{\infty,2} < \min_{L \subseteq \{ 1, \dots, S \}}\frac{\|c_L\|_{\infty}}{4\| c_L\|_2} =: \gamma$ for OMP to find the correct support. So by Corollary~\ref{cor:1} and Lemma~\ref{lem:1} we get
\begin{align*}
\P_S(\| \dico_{I^c}\transp r_J \|_{\infty} & \geq \| \dico_{L}\transp r_J \|_{\infty}) \leq \P_S(\vartheta_I \geq 1/2) + \P_S(\| H_I \|_{\infty,2} \geq \gamma) \\ & \leq 216 K \exp{\left(-\min\left\{\frac{1}{16e^2\|  H \weights \|_{\infty,2}^2 } ,\frac{1}{4\mu} \right\}\right)} + 2K \left( e \frac{\| H \weights \|_{\infty,2}^2}{\gamma^2}  \right) ^{\frac{\gamma^2}{\mu^2}}.
\end{align*}
Owing to the conditions on $\mu$ and $\|H \weights \|_{\infty,2}$ in the theorem, the right hand side does not exceed $\varepsilon$.
\end{proof}
\begin{remark}
Note that for coefficients $c_k \sim \alpha^k$ we can always lower bound $\|c_L\|_{\infty}/\| c_L\|_2 >\sqrt{1- \alpha^2}$. So in the case of uniformly distributed supports ($p_i = S/K$) and a very incoherent dictionary the conditions above reduce to $$ S \mu^2 \lesssim 1- \alpha^2 \quad \mbox{and}\quad S \mu^2 \log K \lesssim 1, $$ which are essentially the same conditions recently derived in \cite{sc18omp} for exactly sparse signals. This is quite surprising, since this new proof is not only shorter but more importantly does not assume random signs of the coefficients but only a random support.
\end{remark}
\subsection{BP}
\noindent
A very popular alternative to the above algorithms is the Basis Pursuit principle. Instead of tackling the NP-hard problem of finding the sparsest solution with greedy methods, it instead aims to solve the convex relaxation
\begin{equation}\label{BP}
\hat{x} = \argmin \| x \|_1  \quad \text{s.t.} \quad y = \dico x.
\end{equation}
The average case performance in the uniform case of this optimisation problem has been extensively studied \cite{tr08, Randall2009, candes2009}. We give a short proof how these results can be transferred to the non-uniform case.
\begin{theorem}\label{them:bp}
Assume that the signals follow the model in \eqref{signal_model}, where the support $I \subseteq\mathbb{K}$ is chosen according to the rejective sampling model with probabilities $p_1, \dots , p_K$ such that $\sum_{i = 1}^K p_i = S$. Further let $W$ denote the corresponding weight matrix. Assume that the hollow Gram-matrix $H = \dico\transp \dico - \mathbb{I}$ satisfies $\| \weights H \weights\|_{2,2} \leq \frac{1}{4 e^2}$. If
\begin{align*}
    \mu \leq \frac{1}{4 \log(220 K/\varepsilon)}, \quad \quad \text{and} \quad \quad
    \| H \weights \|^2_{\infty,2} \leq \frac{1}{16 e^2 \log(220K/\varepsilon)},
\end{align*}
then BP recovers the correct coefficients with probability at least $1 - \varepsilon$.
\end{theorem}
\begin{proof}
We use results for fixed supports such that $\ell_1$ minimisation yields the exact solution \cite{troppl1, Fuchs2004}. Then we show that under the assumptions of the theorem these conditions are satisfied with high probability. 
\begin{proposition}[\cite{troppl1, Fuchs2004}]
Assume $y = \sum_{i \in I}\atom_i c_i \sigma_i$, for some $I\subset \{1,..,K \}$ with $|I|=S$. If
\begin{equation*}
    \|\dico_{I^c} \transp \dico_I (\dico_I\transp \dico_I)^{-1}  \sigma_I\|_{\infty} < 1,
\end{equation*}
then $x$ is the unique solution to the $l_1$-minimisation problem~\eqref{BP}.
\end{proposition}
Now set $M := \dico_{I^c} \transp \dico_I (\dico\transp_I \dico_I)^{-1}$ and $\vartheta_I : = \|\dico_I \transp \dico_I - \mathbb{I} \|$. As usual we note that
\[
\| M \|_{\infty,2} =  \| \dico_{I^c} \transp \dico_I (\dico\transp_I \dico_I)^{-1}  \|_{\infty,2} \leq \| \dico_{I^c} \transp \dico_I \|_{\infty,2} \| (\dico\transp_I \dico_I)^{-1}\|_{2,2} \leq \| H_{I} \|_{\infty,2} \frac{1}{1- \vartheta_I}.
\]
Now Corollary \ref{cor:1} together with applying Hoeffding's inequality to each entry of $M\sigma$ (Lemma~\ref{hoeff}) and Lemma~\ref{lem:1} yield
\begin{align*}
    \P_y \left( \| M \sigma \|_{\infty}   \geq 1  \right) &\leq \P_y\left( \| M \sigma \|_{\infty} \geq 1 \; \middle| \;  \|M\|_{\infty,2}  \leq 2 \gamma \right)  + \P_S\left(\vartheta_I \geq 1/2\right) + \P_S(\| H_I \|_{\infty,2} \geq \gamma) \\
    &\leq 2 K \exp{\left( -\frac{1}{8 \gamma^2}\right)} + 216 K \exp{\left( -\min\left\{\frac{1}{16 e^2 \| H \weights \|_{\infty,2}^2},\frac{1}{4 \mu} \right\}\right)} \\ &\hspace{70mm}+  2K \left( e \frac{\| H \weights \|_{\infty,2}^2}{\gamma^2}  \right) ^{\frac{\gamma^2}{\mu^2}}.
\end{align*}
Setting $\gamma^2 =  \frac{1}{8 \log(220K/ \varepsilon)}$ we see that under the conditions of the Theorem, the failure probability is bounded by $\varepsilon$.
\end{proof}
To illustrate our results we conduct the following small experiment. We take the 2D Haar-Wavelet decomposition of $1000$ randomly chosen normalised patches $y_n$ of size $64\times 64$ from the image \textit{Peppers} before applying a threshold of $\sqrt{\log(d)/d}/6$ for $d = 64^2$ on the coefficients to get a sparse approximation. Counting how often each atom is used we get a proxy for the probability of any atom being in the sparse support $I$ Figure~\ref{fig:dct_subs} (c-d). We denote by $W$ the corresponding weight matrix and by $D$ the vectorised 2D Haar-Wavelet basis. Now we are given two measurement matrices derived from subsampled vectorised 2D-DCT matrices which we denote by $A_1 \in \R^{m \times d }$ and $A_2 \in \R^{m \times d }$. The subsampling pattern is generated by two different subsampling strategies - see Figure~\ref{fig:dct_subs} (a-b). For our experiment we set $m = 512$. We are tasked with solving the minimisation problem 
\begin{equation}
\hat{x} = \argmin \| x \|_1  \quad \text{s.t.} \quad A_i y = A_i D x
\end{equation}
and are given the choice between the two measurement matrices $A_1$ and $A_2$. 
\begin{figure}[ht]
\begin{subfigure}[c]{0.24\linewidth}  
  \centering
  \includegraphics[width=\linewidth]{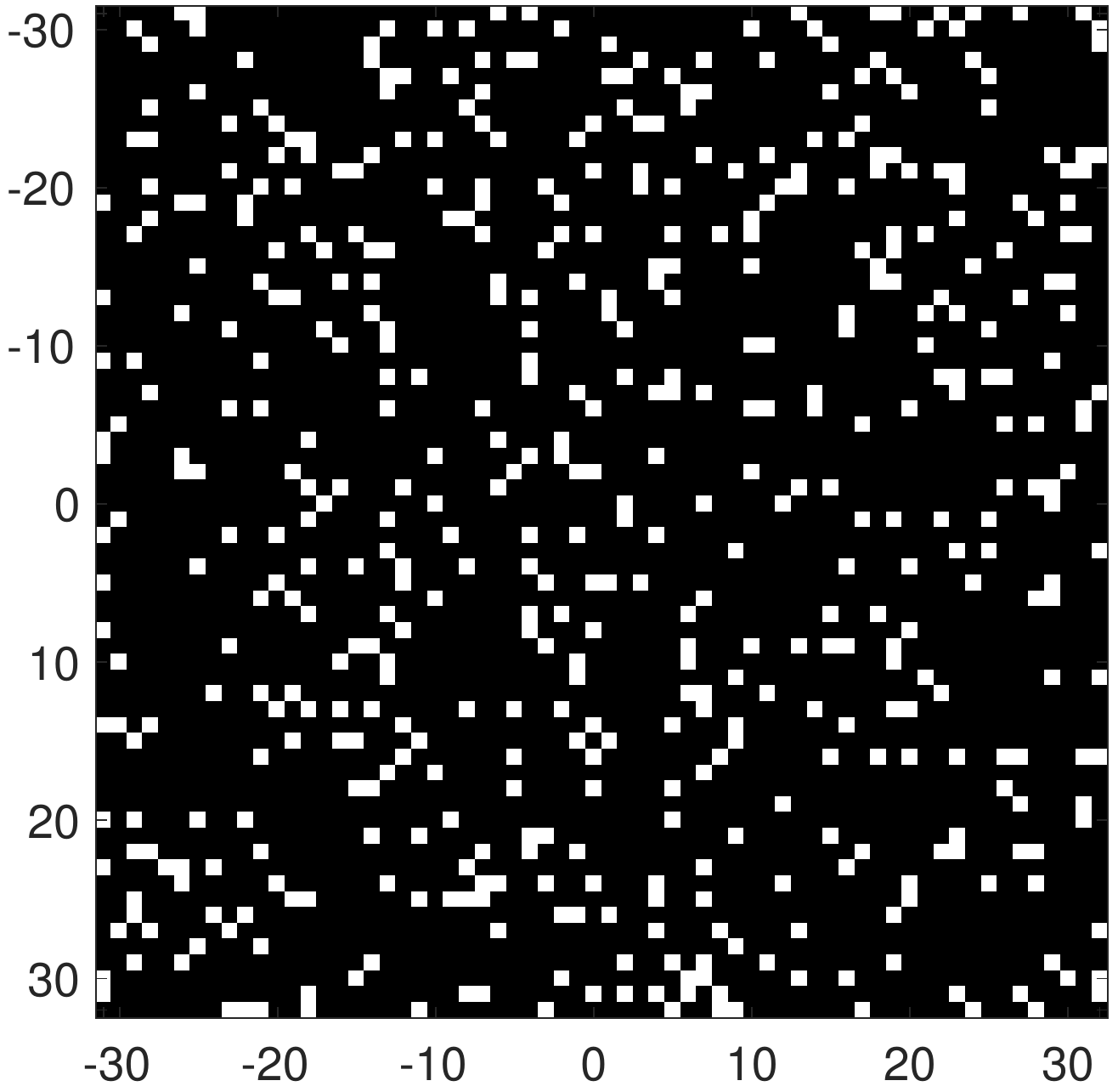}
  \caption{}
\end{subfigure}
\begin{subfigure}[c]{0.24\linewidth}  
  \centering
  \includegraphics[width=\linewidth]{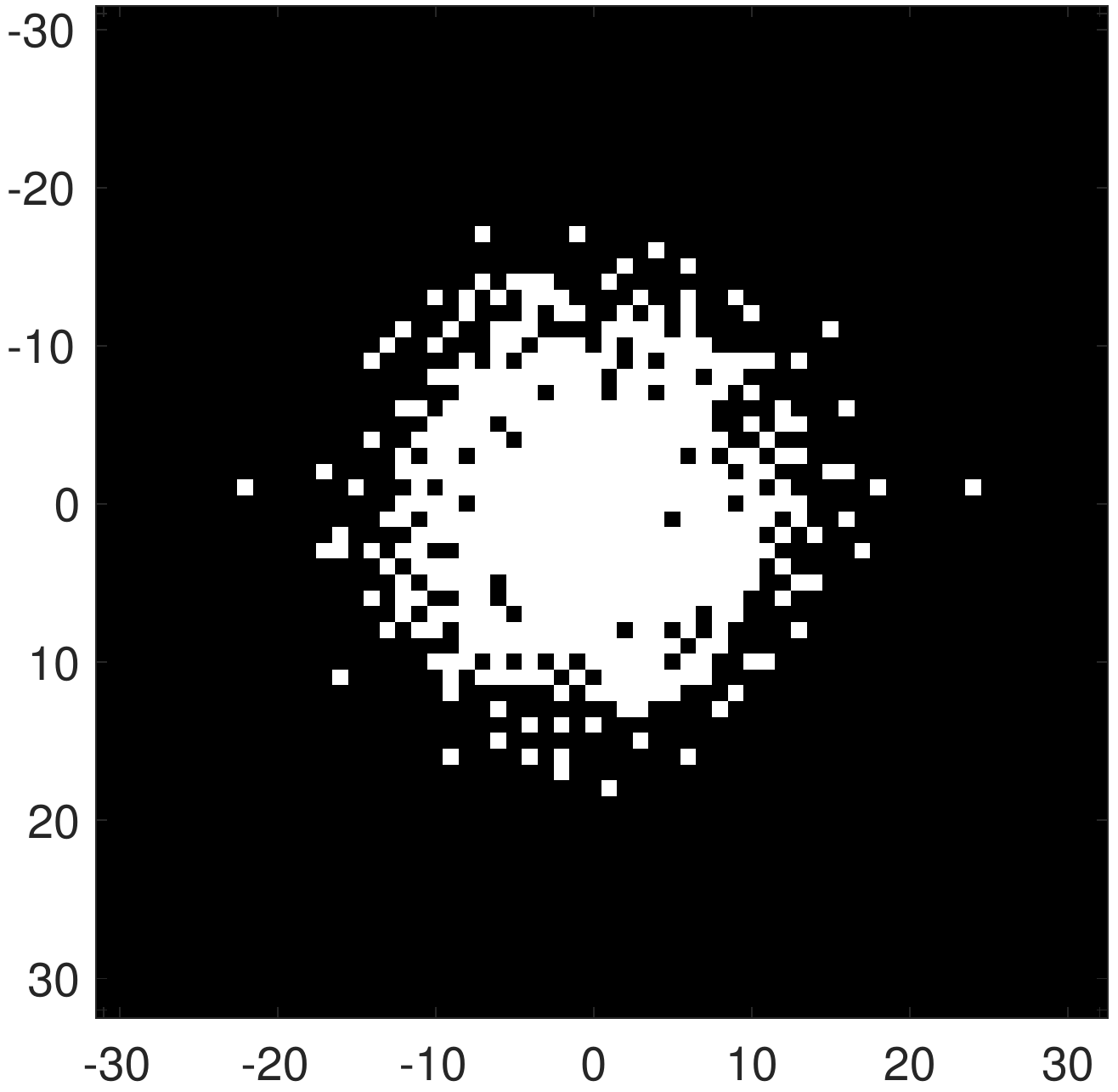}
  \caption{}
\end{subfigure}
\begin{subfigure}[c]{0.25\linewidth}  
  \centering
  \includegraphics[width=\linewidth]{Figure2/weights_Wavelet_6_11_2020.pdf}
  \caption{}
\end{subfigure}
\begin{subfigure}[c]{0.24\linewidth}  
  \centering
  \includegraphics[width=\linewidth]{Figure2/coeffs_Haar_Wavelet_2_6_11_2020.pdf}
  \caption{}
\end{subfigure}
\caption{(a) The K-space $\{(k_1,k_2) : - \sqrt{K}/2 + 1 \leq k_1,k_2 \leq \sqrt{K}/2\}$ with the frequencies used for the measurement matrix $A_1$ (a) and the measurement matrix $A_2$ (b). (c) Expectation of each atom to be in the support (blue) and average expectations for comparison (red) on a log scale. (d) Locations of non-zero coefficients of patches in the 2D-Haar Wavelet Basis.}
\label{fig:dct_subs}
\end{figure}
Our results tell us that as long as the sparse supports of our signals follow the distribution described by the weight matrix $\weights$, we should pick the sensing dictionary $A_i$ that minimises the quantities $\mu$, $\| H \weights\|_{\infty,2}$ and $\|\weights H \weights\|_{2,2}$ (where $H$ is the hollow Gram matrix $A_i D D\transp A_i \transp - \diag(A_i D D\transp A_i \transp)$). Looking at Table~\ref{tbl:dct_subs} columns 1-3 we see that for signals following the distribution specified by $\weights$, our results suggest $A_2$ yields better performance.
\begin{table}[ht]
    \centering
    \begin{tabular}{c c c c c}
       \toprule
        &$\mu$ & $\| H \weights \|_{\infty,2}$ & $\| \weights H \weights \|_{2,2}$ & MSE\\
        \midrule
        $A_1$ & 0.89 &  3.80 & 2.80 & 0.18\\
        $A_2$ & 0.98 & 0.84 & 0.87  & 0.06\\
        \bottomrule
    \end{tabular}
    \caption{The first line corresponds to the uniform subsampling strategy, the second line to the variable density subsampling strategy.}
    \label{tbl:dct_subs}
\end{table}
To test the actual performance, we used BP to recover the coefficients $x_n$ from the set of incomplete measurements $A_i y_n = A_i D x_n$. Note that the coefficients $x_n$ are not sparse, but compressible. Looking at the mean squared error (MSE) $\frac{1}{N}\sum_{n=1}^N\|y_n - D \hat{x}_n \|^2_2$ in Table~\ref{tbl:dct_subs}, we see that even though strictly speaking our theory does not apply here (as these signals are not perfectly sparse) the quantities $\| H \weights\|_{\infty,2}$ and $\|\weights H \weights\|_{2,2}$ seem to be good predictors of average performance for signals where the sparse support (in this case of the biggest entries) follows a distribution specified by a weight matrix $\weights$. 
\section{Sensing dictionaries and preconditioning}\label{sec:sensing}
As an application of our results we construct a sensing dictionary to improve the average performance of a dictionary for thresholding and OMP, given that we know the distribution of supports. We then extend these ideas to BP via preconditioning. \\
In the most general sense a sensing dictionary\footnote{Note that strictly speaking $\Psi\neq \dico$ is not actually a dictionary, as the columns are not normalised.} $\pdico$ for a given dictionary $\dico$ is a matrix of the same size as $\dico$, whose columns satisfy $\ip{\patom_k}{\atom_k} = 1$ for all $k \in \mathbb{K}$. It can be used in greedy algorithms to replace the original dictionary in the atom selection step. Sensing dictionaries improving the worst case performance of OMP and thresholding were first characterised and constructed in~\cite{scva08}. In~\cite{scva07} those ideas were generalised to construct sensing dictionaries that improve the average performance. We extend these average case results to non-uniformly distributed supports to see how the distribution interacts with the structure of the sensing dictionary. \\
The main idea in thresholding and OMP is to determine which atoms to include in the support by looking at the absolute inner products between the signal and the atoms. Using a sensing dictionary changes this step in the thresholding algorithm to
\begin{align*}
    &\text{find} \quad J = \argmax_{|I| = S} \|\pdico_I\transp y\|_1 \quad \text{and}\\
    &\text{reconstruct} \quad x_J = P(\dico_{J})y. 
\end{align*}
For OMP, similarly, the sensing dictionary comes into play when choosing the next atom to add to the support while the residual update step stays the same. Initialising $r_0 = y$ and $J_0 = \emptyset$, for OMP with sensing dictionary $\pdico$ one has to
\begin{align*}
    &\text{find} \quad j = \argmax_k |\ip{\patom_k}{r_i}| \quad \text{and}\\
    &\text{update} \quad J_{i+1} = J_i \cup j \quad \text{resp.} \quad r_{J_{i+1}} = y - P(\dico_{J_{i+1}})y,
\end{align*}
until a stopping criterion is met. Now we will show how to construct a sensing dictionary given knowledge about the distribution of the supports.\\
Assuming that the distribution of our supports follows a Poisson or rejective sampling model with known weight matrix $\weights$, Theorems \ref{lem:thr_sens} and \ref{lem:omp_sens} in the appendix show that a sensing dictionary with good average case performance should ideally minimise $\| (\pdico \transp \dico - \mathbb{I}) \weights \|_{\infty,2 }$. We now try to find $\pdico$ such that this quantity is minimised under the constraint that $\diag(\pdico\transp \dico) = \mathbb{I}$. First note that the quantity $\| (\pdico \transp \dico - \mathbb{I}) \weights  \|_{\infty,2}^2$ is bounded from above by $\| (\pdico \transp \dico - \mathbb{I}) \weights\|_F^2$. Minimising the Frobenius norm instead of the maximum row norm has the big advantage that there exists an easy to find analytic solution. For ease of notation let $ P := W^2$. Following~\cite{scva07} we use Lagrangian multipliers and derive both the objective and the constraint function along $\patom_j$ to get
\[
\frac{d}{d \patom_j}\| \pdico\transp \dico \weights \|_F^2 = \sum_i 2 \ip{\atom_i}{\patom_j}\atom_i p_i = 2 \dico P \dico\transp \patom_j 
\]
\[
\frac{d}{d \patom_j}\ip{\atom_j}{\patom_j} = \atom_j.
\]
So we see that for thresholding and OMP, the sensing dictionary should be set to
\[
\pdico := (\dico P \dico\transp)^{-1}\dico D,
\]
where $D$ is a diagonal matrix s.t. $\ip{\atom_i}{\patom_i} = 1$ for all $i \in \mathbb{K}$. This compares nicely to the result in~\cite{scva07}, where they arrived at $\pdico = (\dico \dico\transp)^{-1}\dico D$ for the special case $p_i = S/K$. This shows how the distribution of coefficients changes the optimal sensing dictionary via the diagonal matrix $P$. Figures \ref{tbl:recovery_rates_th} and \ref{tbl:recovery_rates_omp} show how the performance of thresholding and OMP improves when using sensing dictionaries for various dictionaries and distributions.\\
For BP it is not that simple to use a different sensing dictionary. Instead we use preconditioning, multiplying the original dictionary by an invertible matrix from the left and by a diagonal matrix from the right. Inspired by the heuristic argument above, we set
\[
\pdico = (\dico P \dico\transp)^{-1/2}\dico D^{1/2},
\]
where $D$ is a diagonal matrix s.t. $\ip{\patom_i}{\patom_i} = 1$ for all $i \in \{1, \dots ,K \}$. We then change the BP minimisation problem to
\[
\min \| z\|_1 \quad \text{such that} \quad \Tilde{y} = \pdico z,
\]
where $\Tilde{y} = (\dico P \dico\transp)^{-1/2} y$. This is equivalent to the original optimisation problem, as $D$ is a diagonal matrix with positive entries on its diagonal and $(\dico P \dico\transp)^{-1/2}$ is invertible.\\
\begin{figure}[ht]
\begin{subfigure}[c]{0.5\linewidth}  
      \centering
  \includegraphics[width=0.7\linewidth]{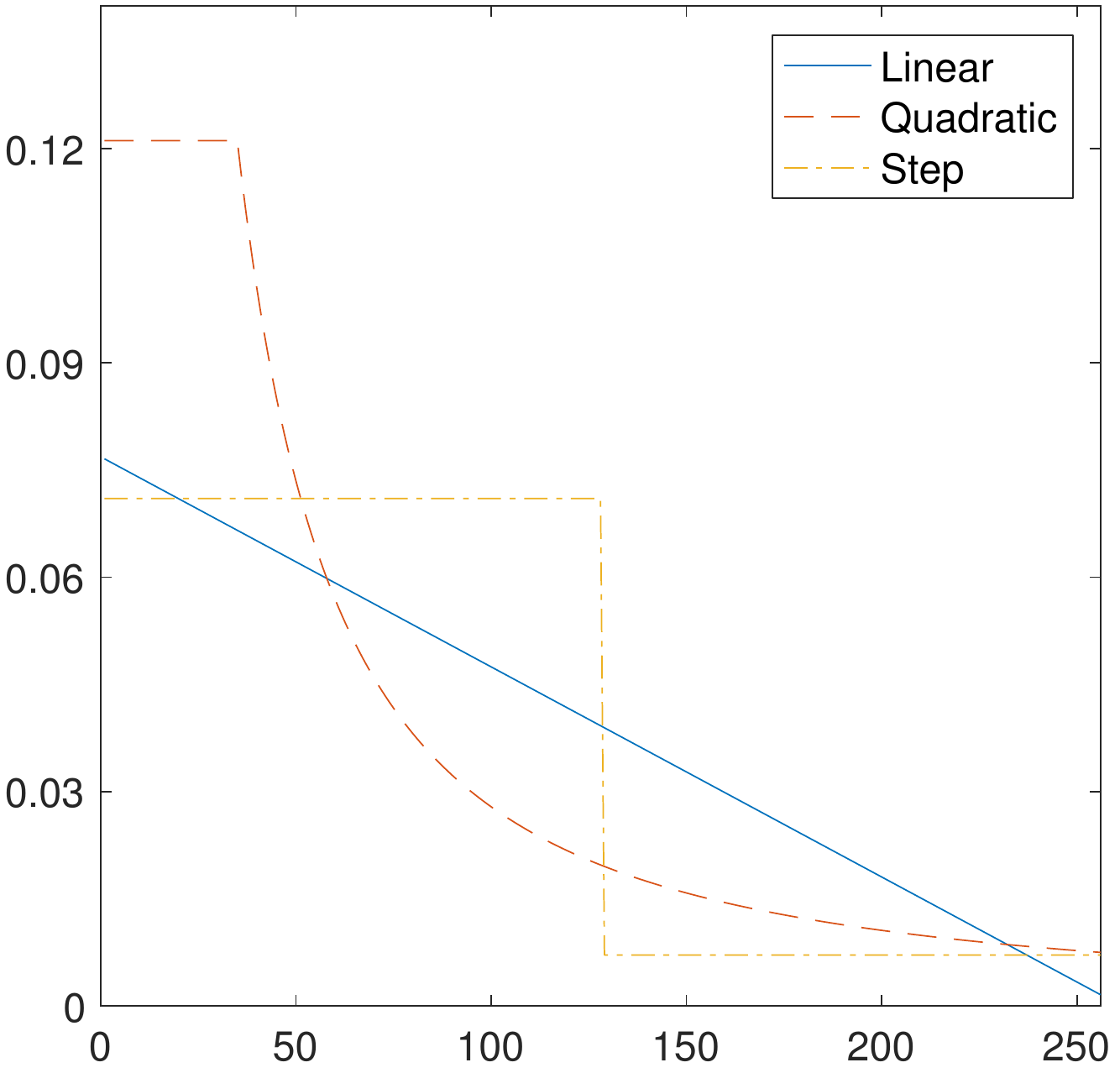}
  \caption{}
\end{subfigure}
\begin{subfigure}[c]{0.5\linewidth}  
      \centering
  \includegraphics[width=0.7\linewidth]{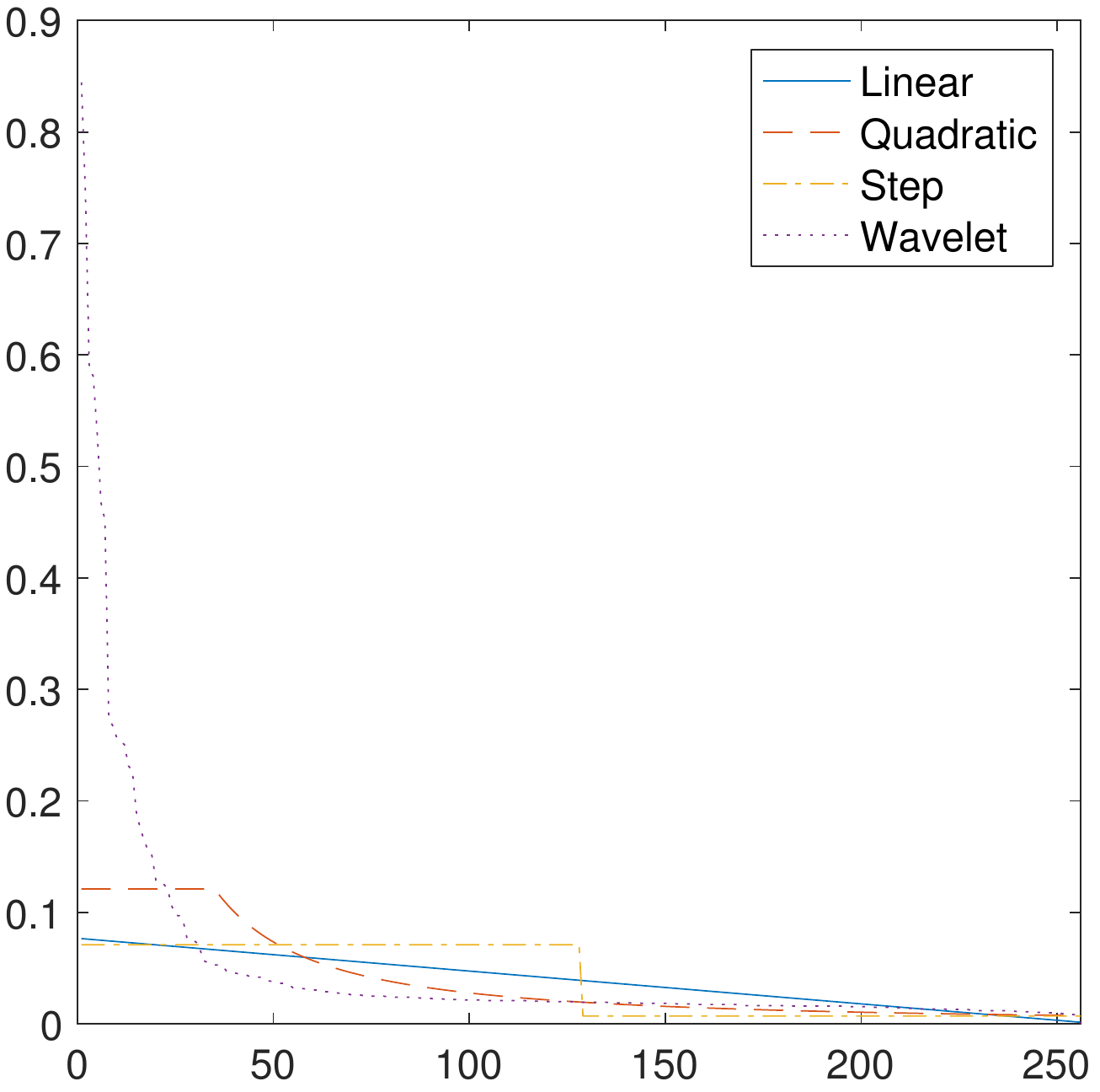}
  \caption{}
\end{subfigure}
\caption{(a) Expectations of the Bernoulli random variables employed in our distribution models. (b) The same plot with the relative frequency of the wavelet coefficients from Figure \ref{wav_coeff} for comparison.}
\label{dist}
\end{figure}
\subsection{Numerical results}
To test the performance of our sensing dictionaries and preconditioning, we conduct the following experiment. We build $2$ dictionaries, each with $256$ atoms of dimension $128$. The columns of the first dictionary are drawn uniformly at random from the unit sphere and the second dictionary is a uniformly subsampled Discrete Cosine Basis with subsequent normalisation. We consider three different distribution models: quadratic, linear and step - see Figure~\ref{dist}. For each distribution model and each support size between $1$ and $80$ we construct $1000$ signals by choosing the support according to the rejective sampling model specified in Section \ref{sec:notations}. The sparse coefficients of $x$ have absolute value one with random signs, i.e. $x_i = \pm 1$ with equal probability. We then compare how often thresholding, OMP and BP can recover the full support when using the original dictionary, the uniform average case sensing dictionary ($P = \mathbb{I}\frac{S}{K}$), and the distribution specific average case sensing dictionary (or the preconditioned matrix for BP). The results for thresholding and OMP are displayed in Table \ref{tbl:recovery_rates_th} and Table \ref{tbl:recovery_rates_omp} respectively. Table \ref{tbl:recovery_rates_bp} shows how the preconditioning changes the recovery rates for BP. As can be seen, incorporating prior knowledge about the distribution of supports into the algorithms improves performance quite significantly for all $3$ algorithms.

\newcolumntype{M}[1]{>{\centering\arraybackslash}m{#1}}

\begin{table}[ht]
    \centering
    \begin{tabular}{c M{43mm}M{43mm}M{43mm}}
       \toprule
         & Linear & Quadratic & Step \\
        \midrule
        \rotatebox[origin=c]{90}{Gaussian} & 
        \includegraphics[width=0.8\linewidth]{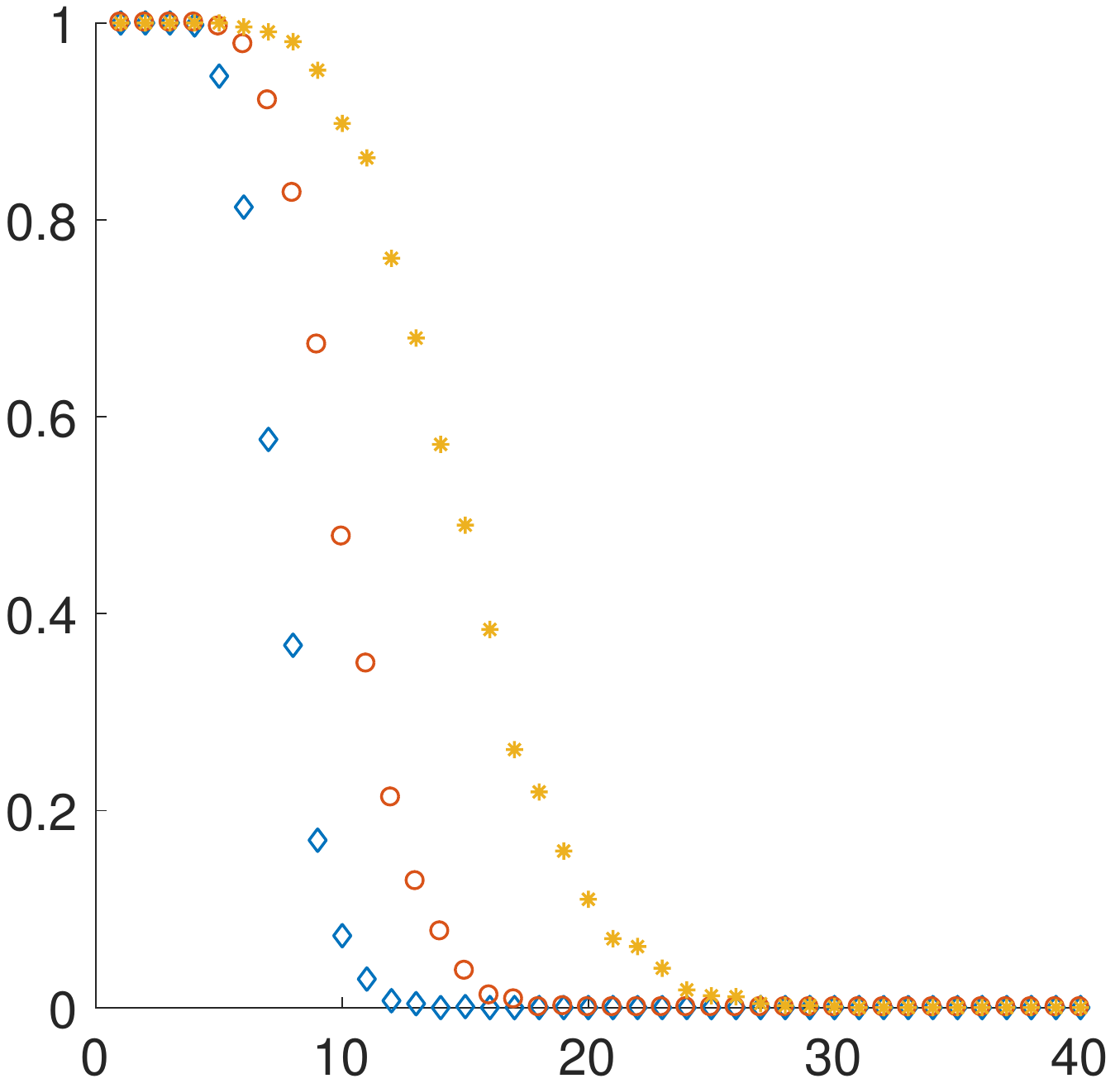}
        & 
        \includegraphics[width=0.8\linewidth]{simulations/Rand_precond_d128K_256_N1000_6_11_2020_th_distribution_1.pdf}
        & 
        \includegraphics[width=0.8\linewidth]{simulations/Rand_precond_d128K_256_N1000_6_11_2020_th_distribution_1.pdf}
        \\
        \rotatebox[origin=c]{90}{Subs. DCT } & 
        \includegraphics[width=0.8\linewidth]{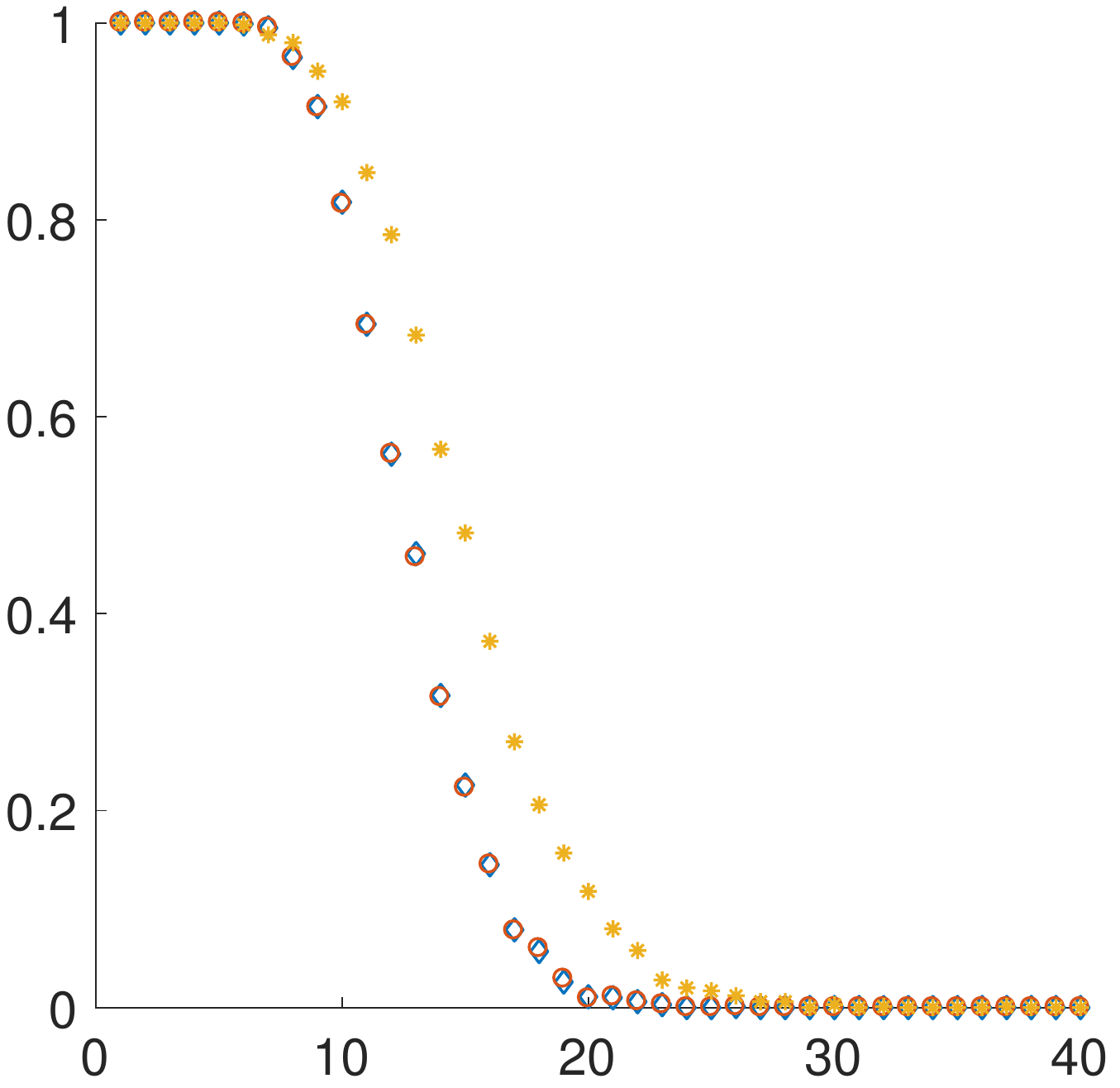}
        & 
        \includegraphics[width=0.8\linewidth]{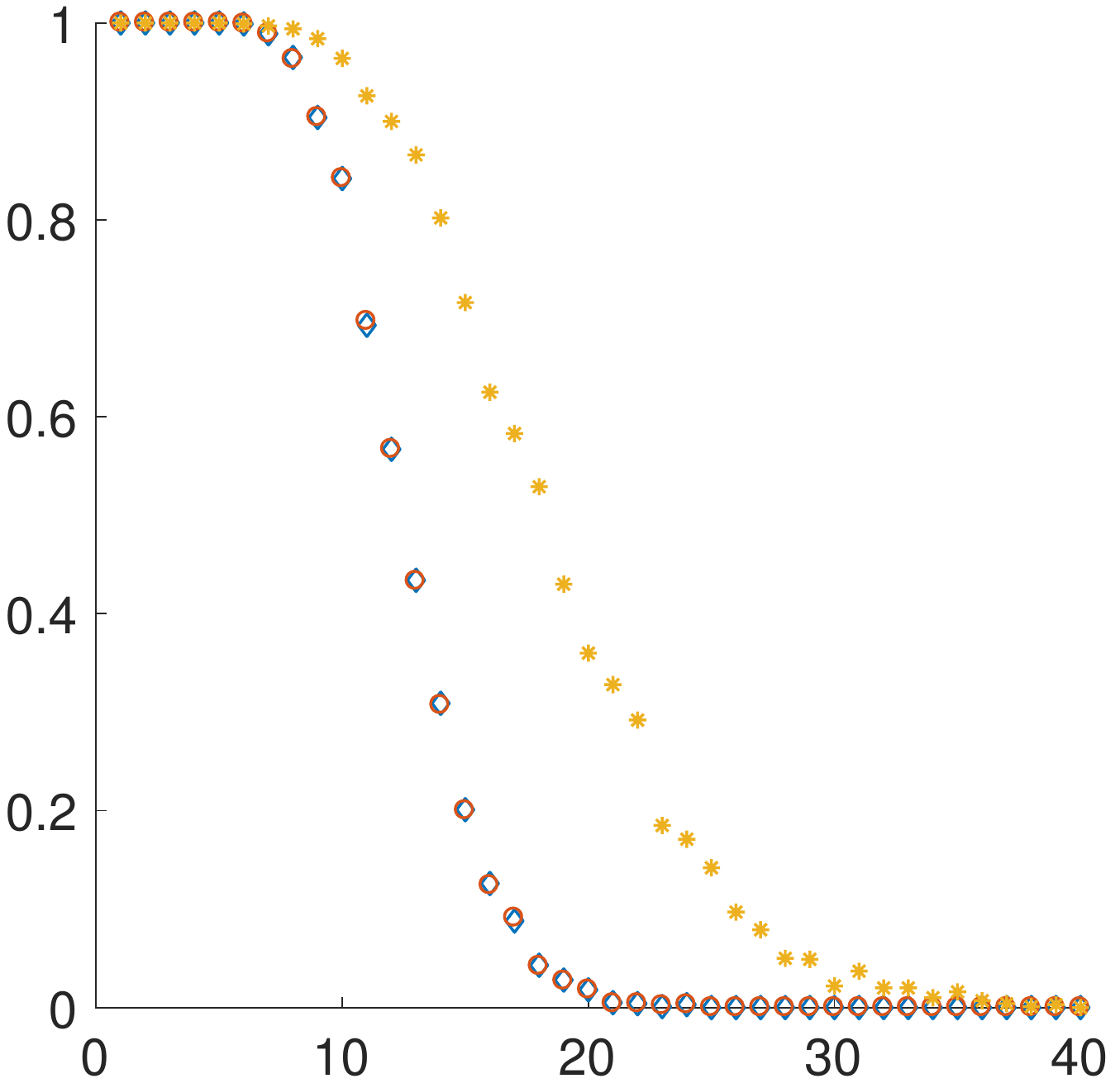}
        & 
        \includegraphics[width=0.8\linewidth]{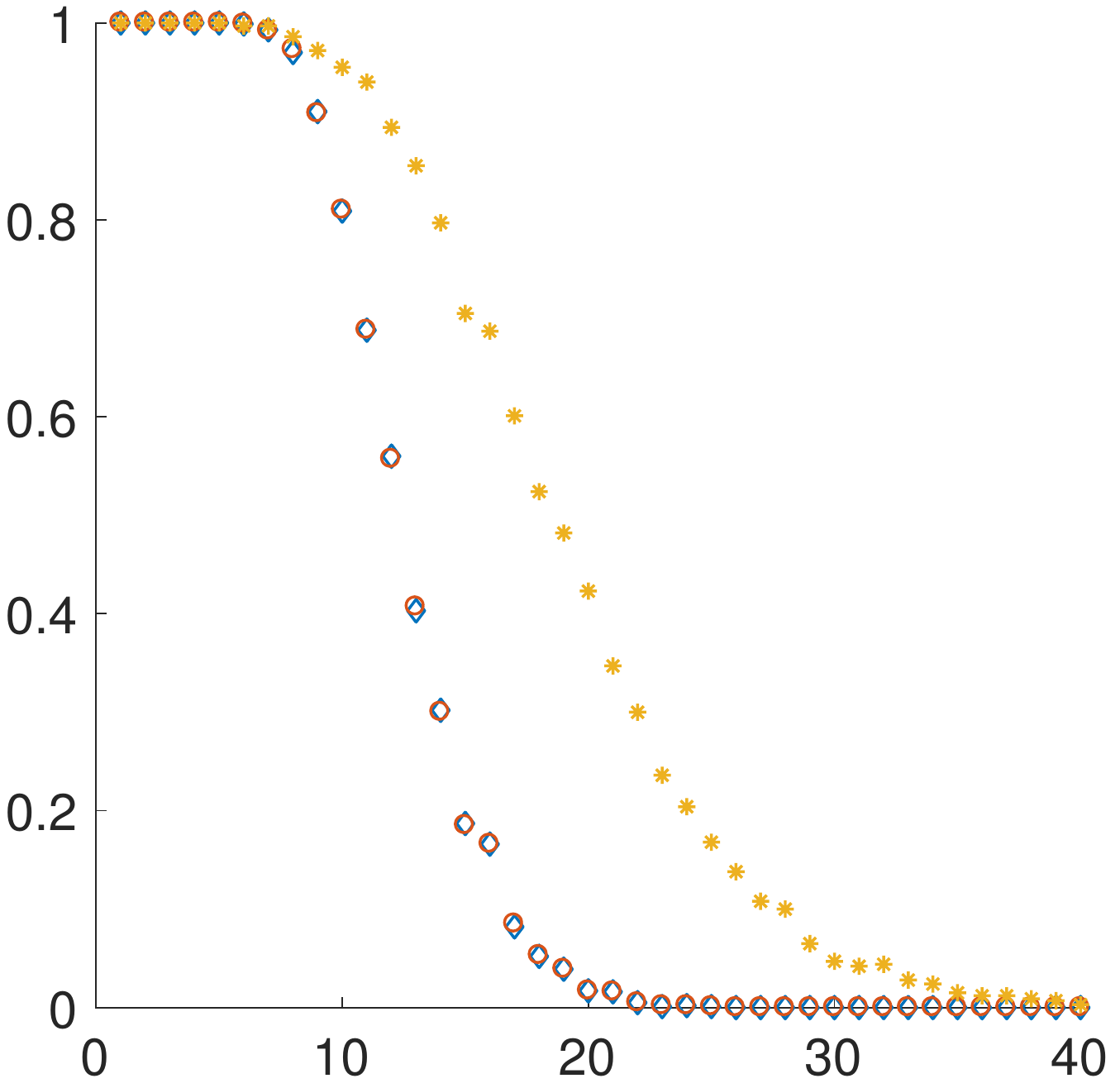} \\
        \bottomrule
    \end{tabular}
    \caption{The recovery rates for thresholding with different sensing dictionaries are plotted on the y-axis and the size of sparse supports on the x-axis. Blue corresponds to no sensing dictionary, red to the uniform average case sensing dictionary and orange to the distribution specific average case sensing dictionary.}
    \label{tbl:recovery_rates_th}
\end{table}

\section{Discussion}
In this paper we have derived concentration inequalities for norms of random subdictionaries with non-uniformly distributed supports. This has allowed us to derive sufficient conditions for sparse approximation algorithms to recover the correct support given that the support of coefficients follows a rejective sampling or Poisson sampling model. We have shown that recovery of signals depends on the structure of the cross-Gram matrix and the distribution of supports, proving that more frequently used atoms should be more incoherent than less frequently used ones. The generalisation from uniformly to non-uniformly distributed supports gives valuable insight into how, in a compressed sensing setup, measurement matrices should be chosen or constructed. For both thresholding and OMP it was shown that using sensing dictionaries that take the distribution of supports into account improves performance. Using precondition to extend this argument to BP, it was also shown that prior knowledge about the distribution leads to improved performance for BP as well.\\
Our next goal is to use these results to prove convergence of dictionary learning algorithms for signals where the atoms of the generating dictionary are not equally used - as seems to happen in practice. Not only should this show that the more frequently used atoms converge faster, but it should also give insights how to best estimate the size of the generating dictionary.

\acks{This work was supported by the Austrian Science Fund (FWF) under Grant no.~Y760. Finally, many thanks go to Elisabeth Schneckenreiter for proof-reading the manuscript.}

\begin{table}[ht]
    \centering
    \begin{tabular}{c M{43mm}M{43mm}M{43mm}}
       \toprule
         & Linear & Quadratic & Step \\
        \midrule
        \rotatebox[origin=c]{90}{Gaussian} & 
        \includegraphics[width=0.8\linewidth]{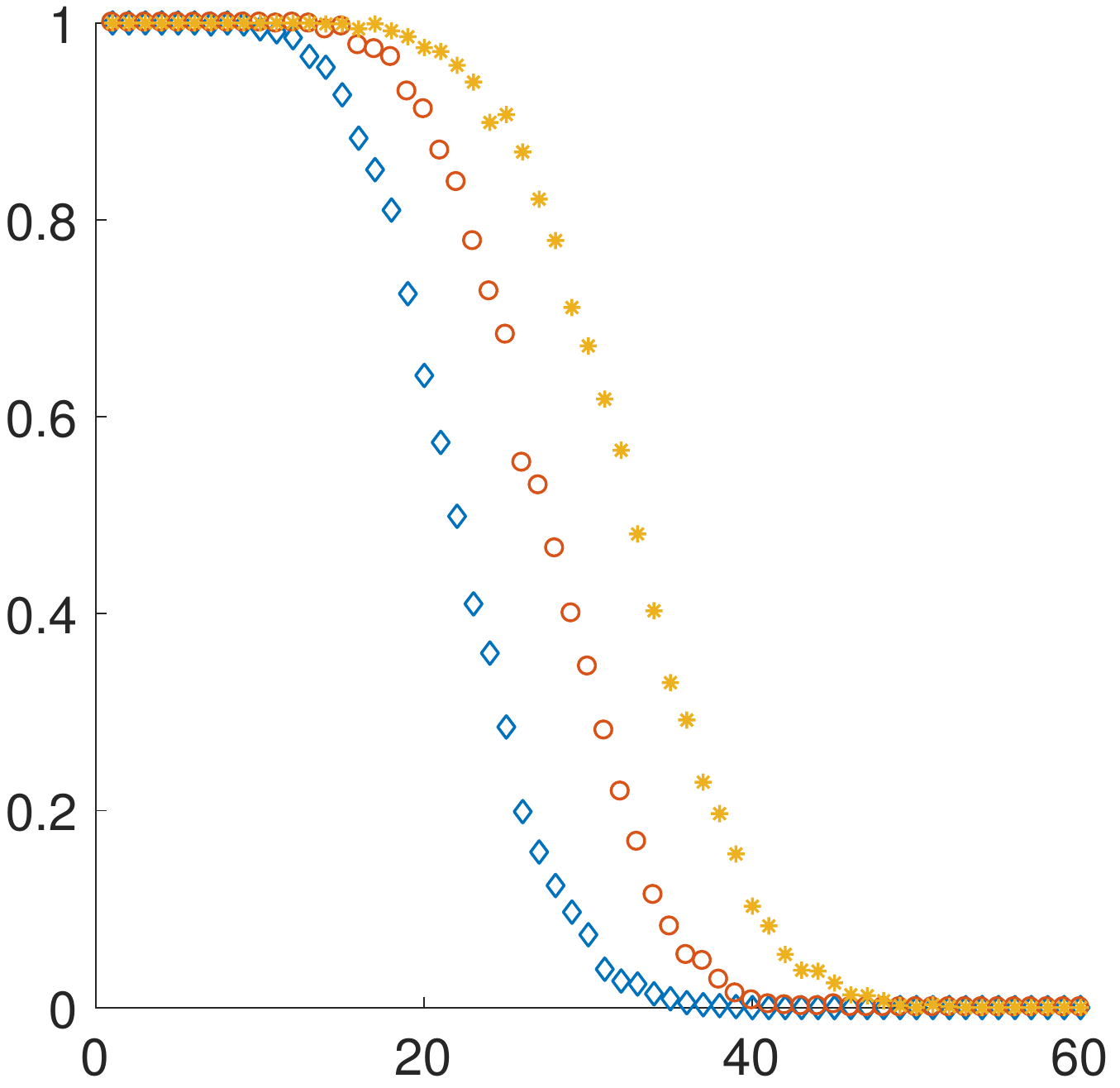}
        & 
        \includegraphics[width=0.8\linewidth]{simulations/Rand_precond_d128K_256_N1000_6_11_2020_omp_distribution_1.pdf}
        & 
        \includegraphics[width=0.8\linewidth]{simulations/Rand_precond_d128K_256_N1000_6_11_2020_omp_distribution_1.pdf}
        \\
        \rotatebox[origin=c]{90}{Subs. DCT } & 
        \includegraphics[width=0.8\linewidth]{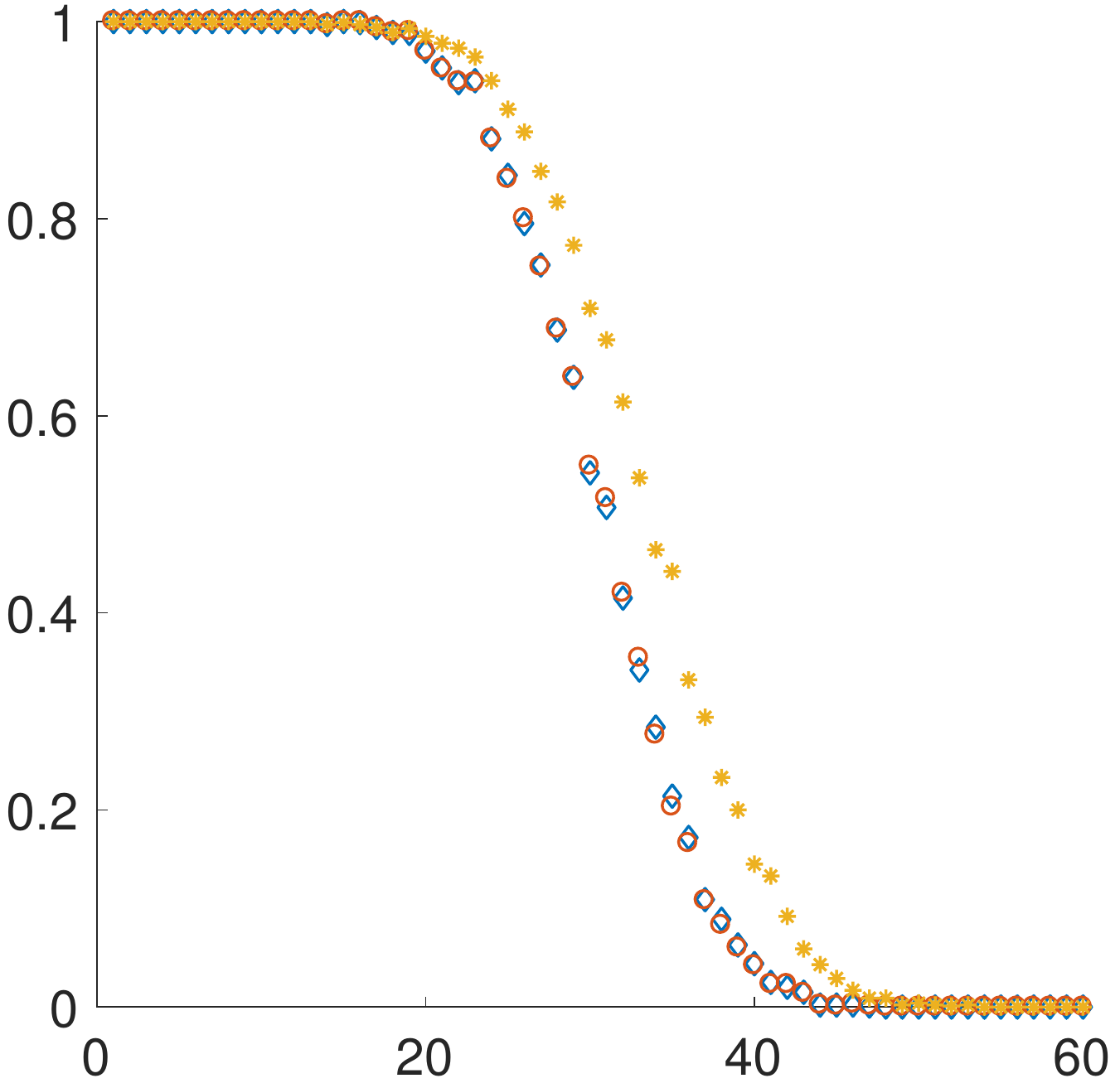}
        & 
        \includegraphics[width=0.8\linewidth]{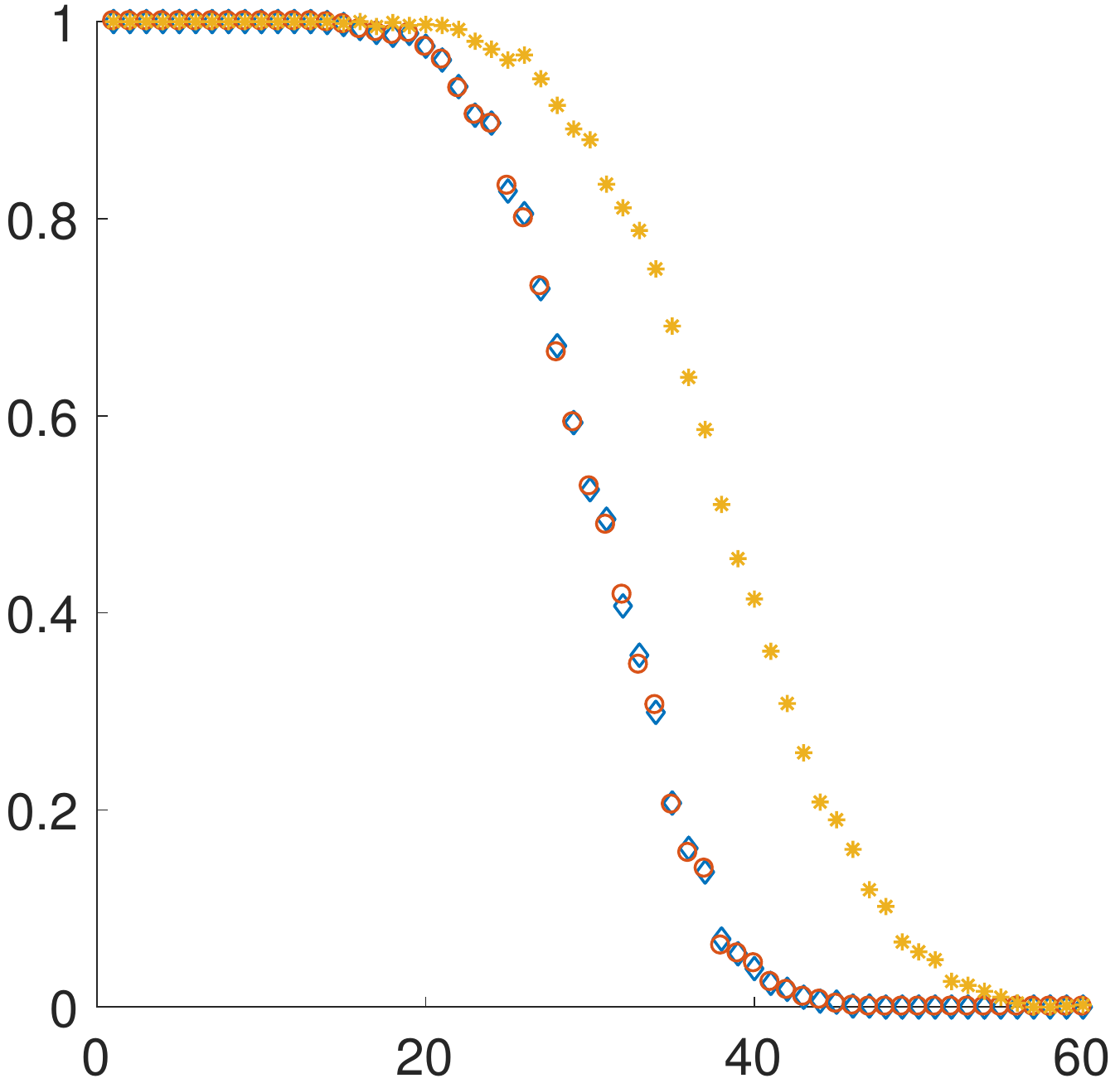}
        & 
        \includegraphics[width=0.8\linewidth]{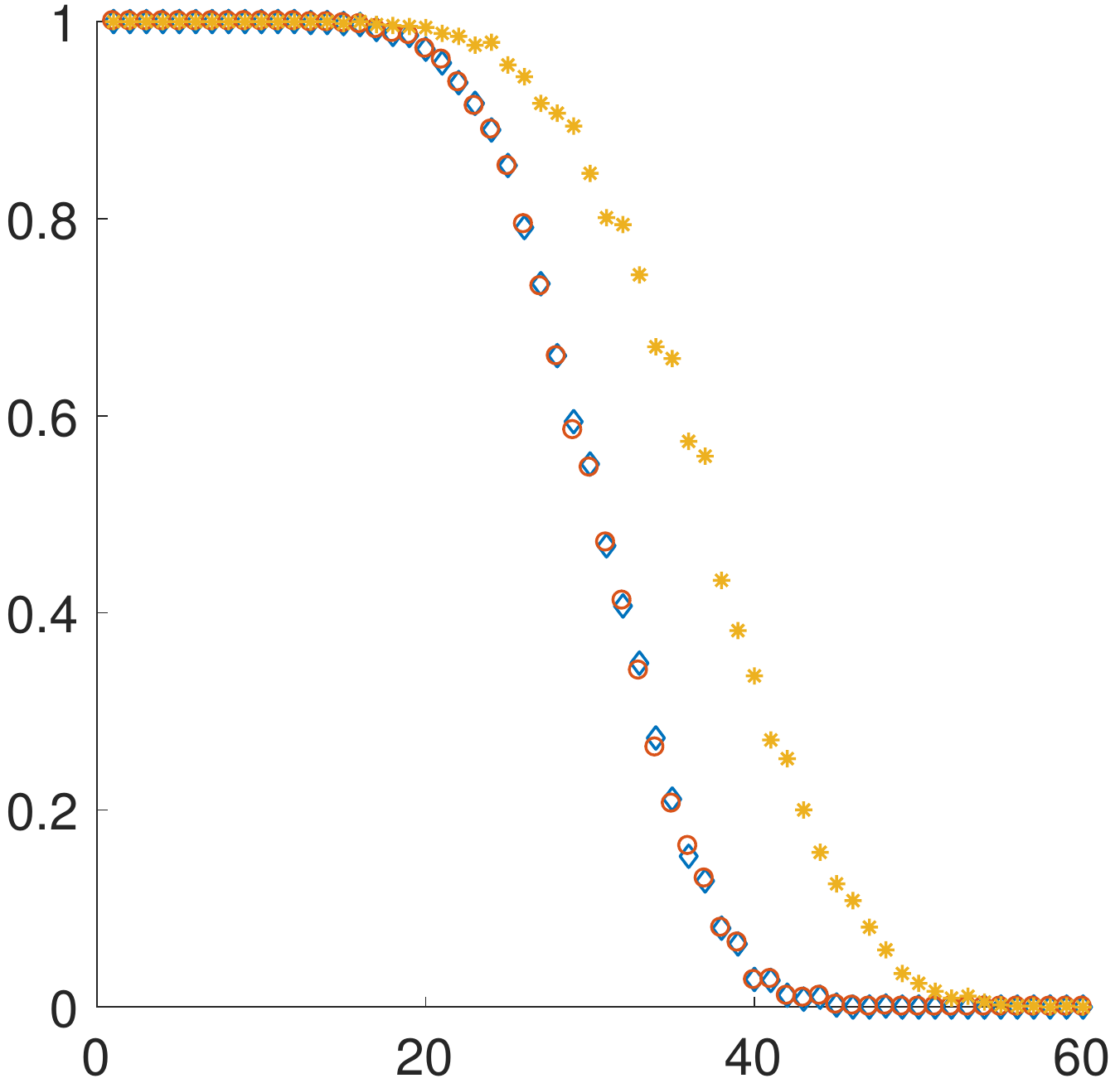} \\
        \bottomrule
    \end{tabular}
    \caption{The recovery rates for OMP with different sensing dictionaries are plotted on the y-axis and the size of sparse supports on the x-axis. Blue corresponds to no sensing dictionary, red to the uniform average case sensing dictionary and orange to the distribution specific average case sensing dictionary.}
    \label{tbl:recovery_rates_omp}
\end{table}

\begin{table}[ht]
    \centering
    \begin{tabular}{c M{43mm}M{43mm}M{43mm}}
       \toprule
         & Linear & Quadratic & Step \\
        \midrule
        \rotatebox[origin=c]{90}{Gaussian} & 
        \includegraphics[width=0.8\linewidth]{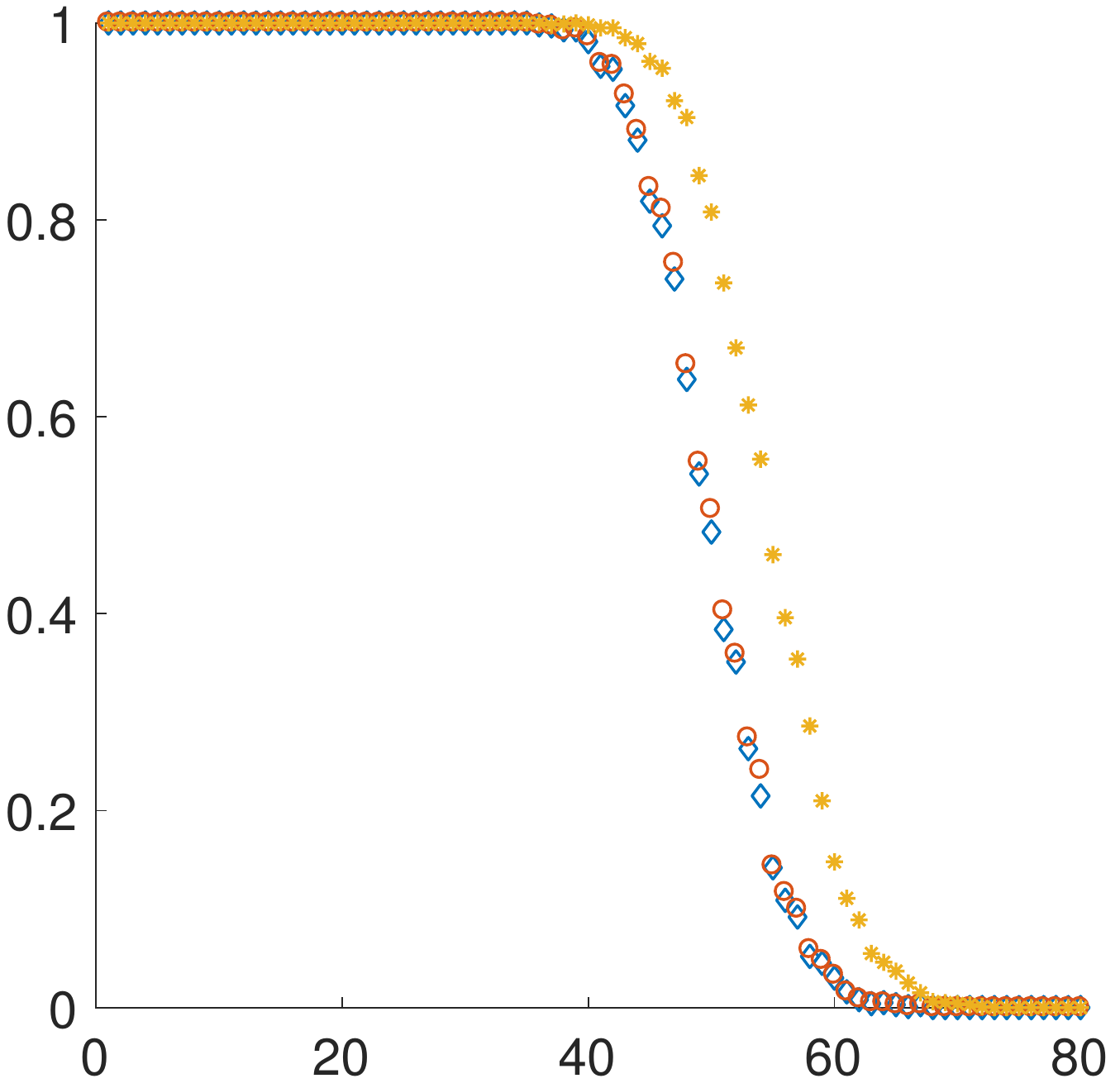}
        & 
        \includegraphics[width=0.8\linewidth]{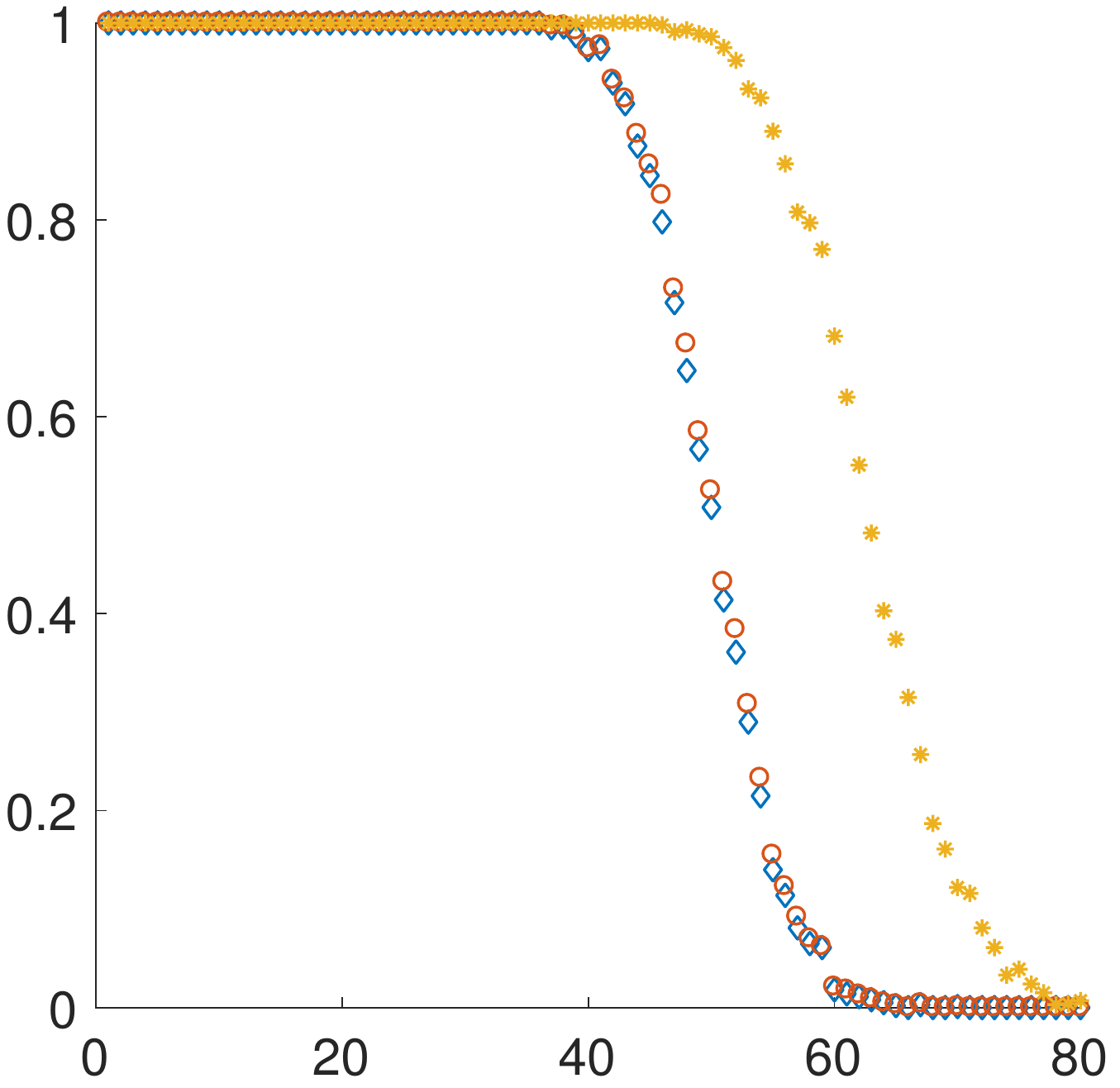}
        & 
        \includegraphics[width=0.8\linewidth]{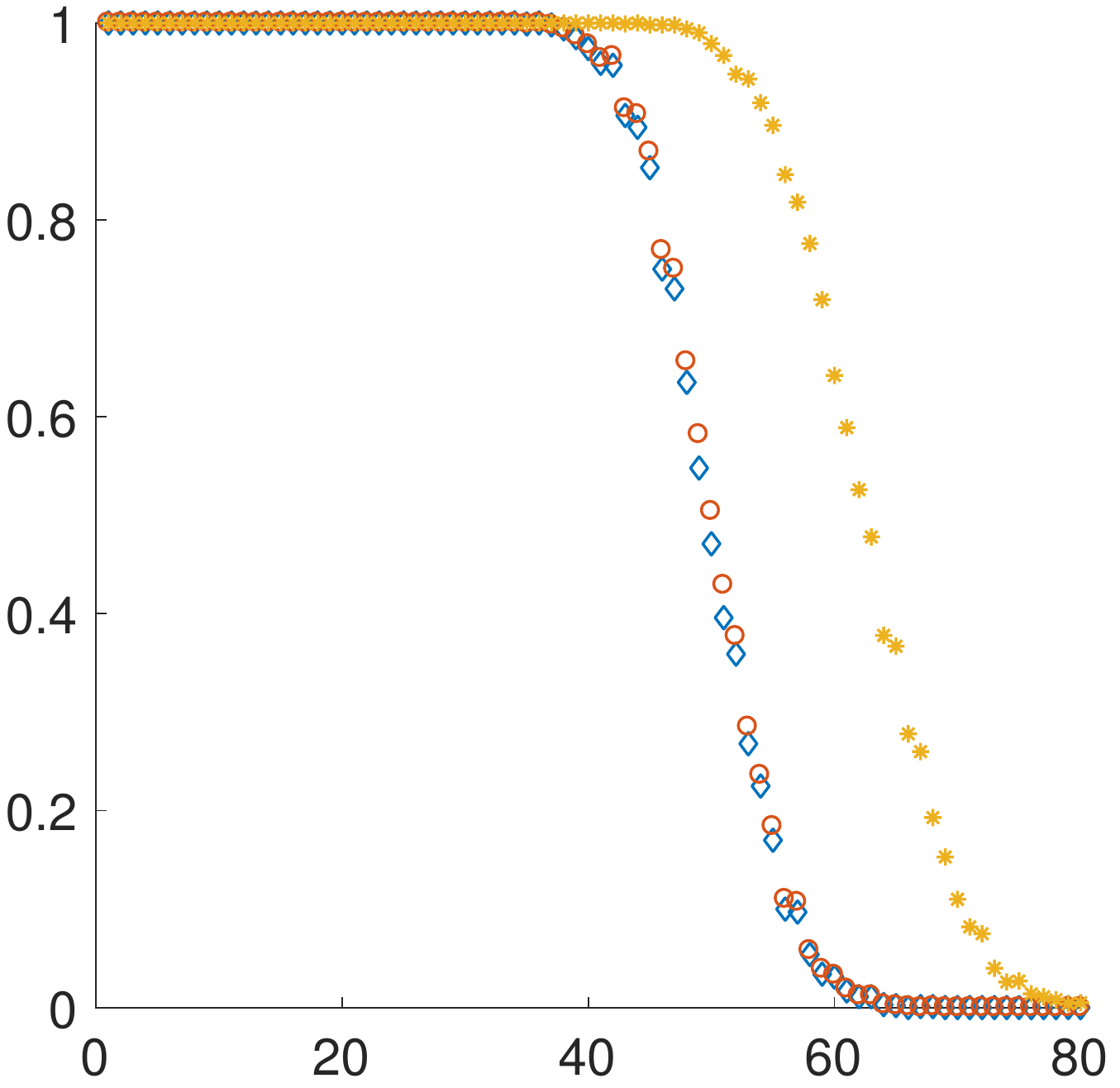}
        \\
        \rotatebox[origin=c]{90}{Subs. DCT } & 
        \includegraphics[width=0.8\linewidth]{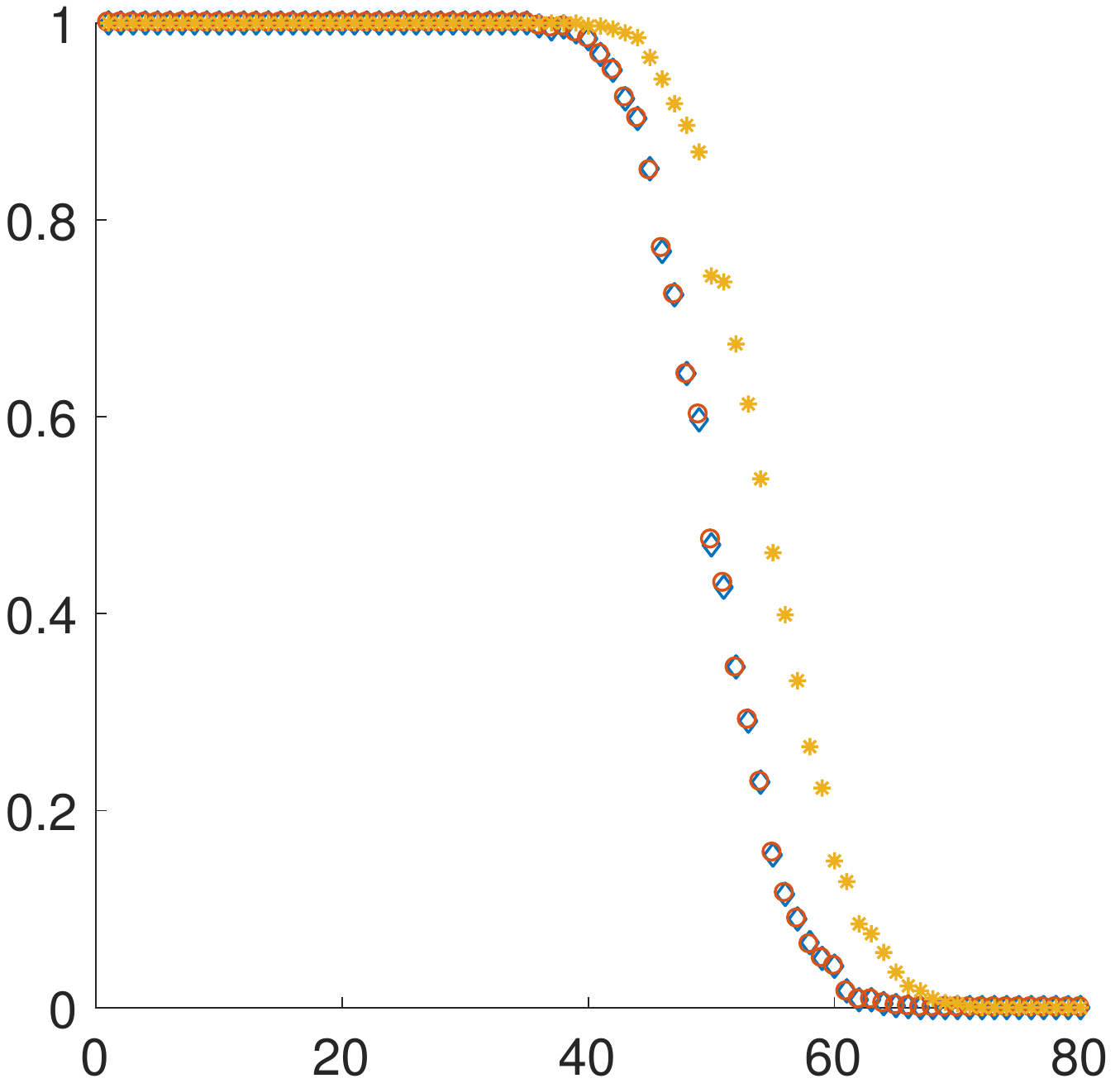}
        & 
        \includegraphics[width=0.8\linewidth]{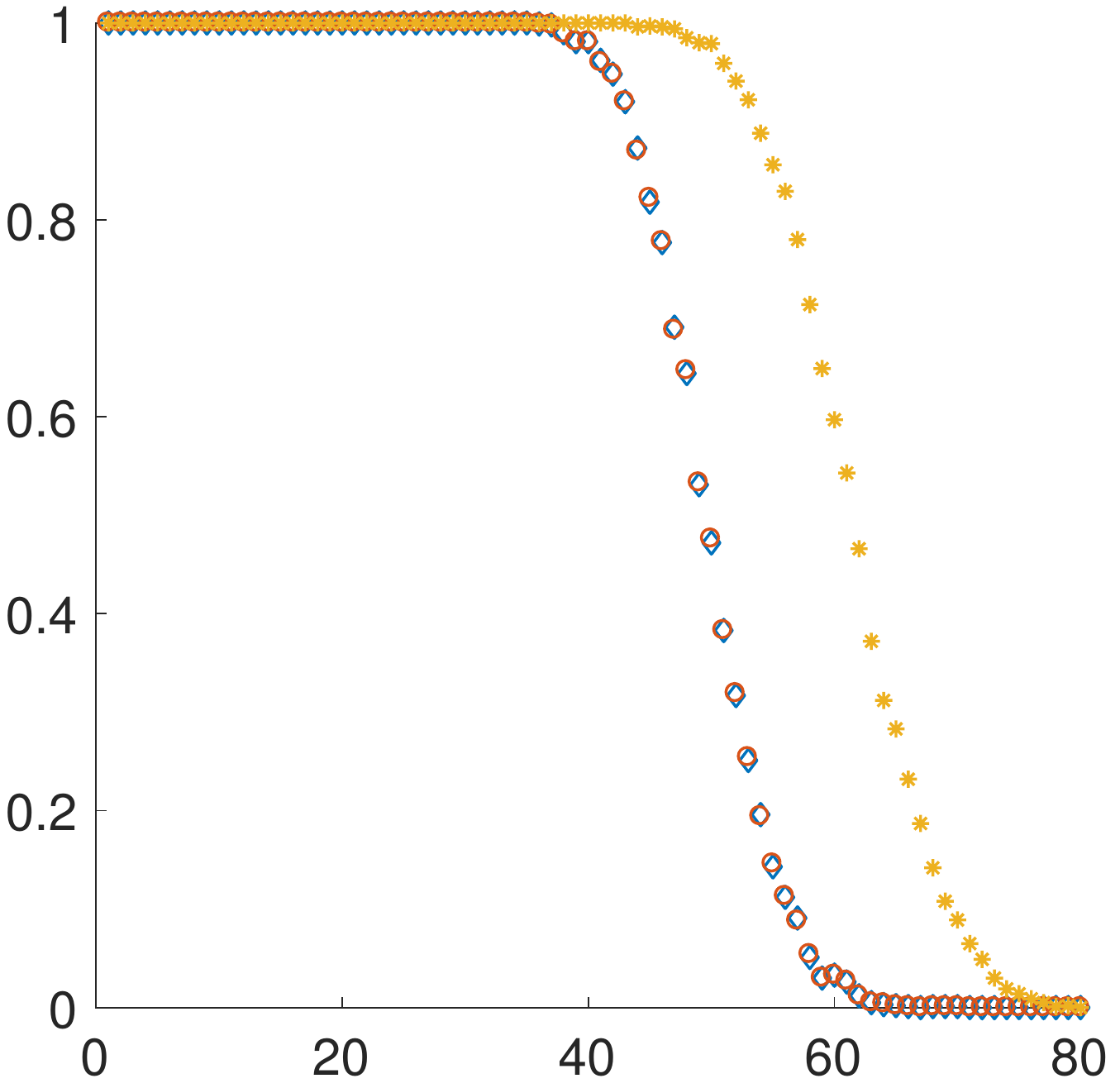}
        & 
        \includegraphics[width=0.8\linewidth]{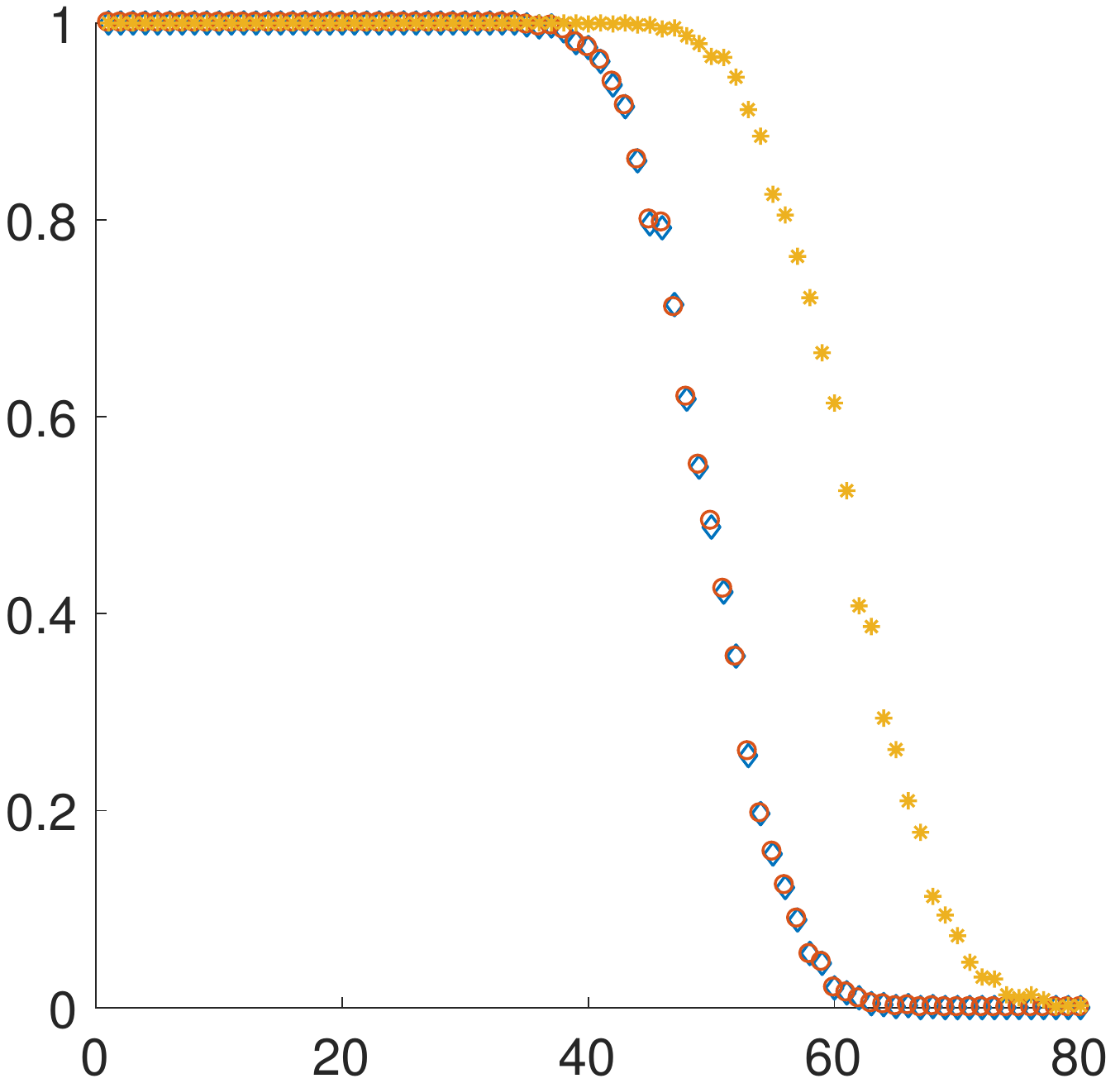} \\
        \bottomrule
    \end{tabular}
    \caption{The recovery rates for BP with different preconditioning strategies are plotted on the y-axis and the size of sparse supports on the x-axis. Blue corresponds to the original $\ell_1$ minimisation problem, red corresponds to preconditioning in the uniform case and orange corresponds to preconditioning with the correct weights.}
    \label{tbl:recovery_rates_bp}
\end{table}

\newpage
\appendix
\section{Proof of Theorem \ref{them:1}}\label{poof_them1}
The proof follows the one that appeared in Chr\'etien and Darses~\cite{chda12} with some minor changes to account for the non-uniformly distributed supports and the extension to non-symmetric matrices. We start with an argument that lets us decouple the random variables selecting the rows and columns. This is crucial for the application of concentration inequalities for sums of independent random matrices later in the proof.
\begin{proposition}
Let $H \in \R^{d \times K}$ be some matrix. Assume $I \subseteq \mathbb{K}$ is chosen according to the Poisson sampling model with probabilities $p_1, \dots , p_K$ such that $\sum_{i = 1}^K p_i = S$. Further let $W$ denote the corresponding weight matrix. Then, for all $r \geq 0$
\begin{equation*}
    \P\left ( \| \proj H \proj \| \geq r  \right) \leq 36 \; \P \left ( \| \proj H \proj' \| \geq r/2  \right),
\end{equation*}
where $\proj'$ is an independent copy of $\proj$.
\end{proposition}
\begin{proof}
Let $\eta_i$ for $1 \leq i \leq K$ be a series of i.i.d. Rademacher random variables. We follow the approach of Chr\'etien/Darses~\cite{chda12} and Tropp~\cite{tropp08} who refer to Bourgain/Tzafriri~\cite{botz87} and de la Pe\~{n}a/Gin\'{e}~\cite{DecouplingBook}. We define
\begin{equation*}
    Z = Z(\eta,\delta) : = \sum_{i \neq j} (1 - \eta_i \eta_j)\delta_i \delta_j \Hij .
\end{equation*}
Setting $Y = \sum_{i \neq j} \delta_i \delta_j \Hij \eta_i \eta_j$, we can write
\begin{equation*}
    Z = \proj H \proj - Y.
\end{equation*}
\ifthenelse{\boolean{arxiv}}{
Recall the Hahn-Banach Theorem.
\begin{theorem}[Hahn-Banach]
Let $X$ be a real vector space and p a sublinear functional on X. Let f be a linear functional defined on a subspace $A \subset X$, and satisfying $f(a) \leq p(a)$ for all $a \in A$. Then there exists a linear functional $\Tilde{f}$ on X satisfying
\begin{align*}
    &\Tilde{f}(a) = f(a) \quad \text{for all} \quad a \in A \quad \text{and}\\ 
    &\Tilde{f}(x) \leq p(x) \quad \text{for all} \quad x \in X.
\end{align*}
\end{theorem}
}{}
From now on we work conditional on a choice of $I$ (i.e. we fix our sequence $\delta_i$, therefore the support set $I$ and the entries of $\proj$ are fixed as well). Denote by $A = \left\{\lambda  \proj H \proj \; \middle| \;  \lambda \in \R \right\}$ the subspace generated by $ \proj H \proj $ and define a linear form $f(\lambda  \proj H \proj) = \lambda \| \proj H \proj\| $ on this subspace. By definition we have $f(a) \leq \| a\| = : p(a)$ for all $a \in A$, where the properties of the operator norm imply that $p$ is a sublinear functional. Thus the Hahn-Banach Theorem gives us the existence of a linear functional $\Tilde{f}$ satisfying
\[
\Tilde{f}( \proj H \proj) = f( \proj H \proj) = \| \proj H \proj \|\]
and 
\[
\Tilde{f}(Z) \leq \| Z  \|.
\]
Using the linearity of $\Tilde{f}$ and that $Y$ is symmetric around $0$ we get
\begin{align*}
    \P_{\eta} \left( \| Z \| \geq \| \proj H \proj \|  \right) & = \P_{\eta} \left( \| Z \| \geq \Tilde{f}( \proj H \proj)  \right) \\ & \geq \P_{\eta} \left( \Tilde{f}(Z) \geq \Tilde{f}( \proj H \proj) \right) \\ & = \P_{\eta} \left( \Tilde{f}(-Y) + \Tilde{f}( \proj H \proj) \geq \Tilde{f}( \proj H \proj)  \right) = \P_{\eta} \left( \Tilde{f}(Y) \geq 0  \right),
\end{align*}
where again by linearity of $\Tilde{f}$ we have
\begin{equation*}
    \Tilde{f}(Y) = \sum_{\substack{i \neq j \\ i,j \in I}}\Tilde{f}(\Hij)\eta_i \eta_j  = \sum_{\substack{i>j \\ i,j \in I}}\left[\Tilde{f}(\Hij)+\Tilde{f}(\Hji)\right]\eta_i \eta_j.
\end{equation*}
So we see that $\Tilde{f}(Y)$ is a homogeneous Rademacher chaos of order $2$. \ifthenelse{\boolean{arxiv}}{For ease of notation write $\xi := \Tilde{f}(Y)$. As $\xi$ is a centered real random variable we can write $\E[|\xi|] = 2\E[\xi \I_{\xi>0}]$ and a simple application of H\"olders inequality yields
\begin{equation*}
    \E[|\xi|]^2 = 4\E[\xi \I_{\xi>0}]^2 \leq 4 \P \left( \xi >0\right)\E[\xi^2].
\end{equation*}
Write $\E[ \xi^2] = \E[\xi^{2/3}\xi^{4/3}]$ and apply H\"olders inequality again with $p=\frac{3}{2}$ and $q = 3$ to get
\[
\E[\xi^2] \leq \E [|\xi|]^{\frac{2}{3}} \E[\xi^4]^{\frac{1}{3}}.
\]
Putting the above together we arrive at
\begin{equation*}
    \P(\xi >0 ) \geq \frac{1}{4}\frac{\E[|\xi|]^2}{\E[\xi^2]} \geq \frac{1}{4}\frac{\E[\xi^2]^2}{\E[\xi^{4}]}.
\end{equation*}
Since $\xi$ is a homogeneous Rademacher chaos of order $2$ we can apply Lemma 2.1 of Chr\'etien and Darses~\cite{chda12}, which states
\[
\frac{\E[\xi^2]^2}{\E[\xi^{4}]} \geq \frac{1}{9}.
\]
So}{Following the arguments in Chr\'etien and Darses~\cite{chda12} we get}
\begin{equation*}
    \P_{\eta} \left( \| Z\| \geq \|  \proj H \proj \| \right)  \geq \frac{1}{36}.
\end{equation*}
Multiplying both sides with $\I_{ \{ \|  \proj H \proj \| \geq r \} }$ and taking the expectation w.r.t. to $I$ we get
\begin{equation*}
    \P( \|  \proj H \proj \| \geq r) \leq 36 \;  \P( \| Z\| \geq r).
\end{equation*}
Now by the same argument as in Tropp~\cite{tropp08}, Proposition 2.1 there exists a $\bar\eta \in \{-1,1 \}^K$ s.t.
\begin{equation*}
        \P \left( \| Z \| \geq r  \right) = \E \left[ \E \left( \I_{\{ \|Z(\eta,\delta) \| \geq r\}} \; \middle| \;  \eta \right)\right]  \leq \E \left(\I_{\{ \| Z(\bar\eta,\delta) \| \geq r \}}\right) = \P \left( \| Z(\bar\eta,\delta) \| \geq r\right).
\end{equation*}
Setting $T = \{ i : \bar\eta_i =1\}$, we see by the definition of $Z$
\begin{equation*}
    Z(\bar\eta,\delta) = 2 \sum_{j \in T, k \in T^c}\delta_j \delta_k \Hjk +2 \sum_{j \in T^c, k \in T}\delta_j \delta_k \Hjk= 2 \sum_{j \in T, k \in T^c}\delta_j \delta_k (\Hjk+ \Hkj).
\end{equation*}
Now we can do the decoupling. As $\delta_i$ for $i \in T$ are independent from $\delta_j$ for $j \in T^c$ we can replace $\delta_j$ for $j \in T^c$ with $\delta'$ which is an independent copy of $\delta$. Thus
\begin{equation*}
    \P\left( \left\| Z \right\| \geq r\right) \leq \P\left( \| \sum_{j \in T, k \in T^c} \delta_j \delta'_k (\Hjk+\Hkj) \| \geq r/2\right),
\end{equation*}
Note that (after reordering) this matrix is of the form $\begin{pmatrix}
    0 & A \\
    B & 0 
    \end{pmatrix}$, where $A$ corresponds to $ \sum_{j \in T, k \in T^c} \delta_j \delta'_k \Hjk$ and $B$ to $\sum_{j \in T, k \in T^c} \delta_j \delta'_k \Hkj$. The operator norm of this reordered matrix satisfies 
\begin{equation*}
     \left \| \begin{pmatrix}
    0 & A \\
    B & 0 
    \end{pmatrix}\right\|^2 =   \left \| \begin{pmatrix}
    B\transp B & 0 \\
    0 & A\transp A 
    \end{pmatrix}\right\| = \max\{ \|A \|^2, \| B\|^2 \}.
\end{equation*}
As the spectral norm of a submatrix is always less than or equal to the spectral norm of the whole matrix we get by reintroducing the missing entries
\begin{equation*}
    \P( \| Z\| \geq r) \leq \P( \| \proj H \proj' \| \geq r/2).
\end{equation*}
Putting everything together yields the desired result.
\end{proof}
Now we are in a position to apply concentration inequalities for sums of independent random matrices. For that recall the Matrix Chernoff inequality, which can be found in~\cite{tr12}. 
\begin{theorem}[Matrix Chernoff inequality~\cite{tr12}]
Let $X_1,...,X_K$ be independent, symmetric and positive semi-definite random matrices taking values in $\R^{d \times d}$. Now let $B,m > 0$ and assume that for all $1 \leq k \leq K$
\begin{align*}
        \| X_k \| \leq B  \quad \text{and} \quad \| \sum_{k=1}^K\E X_k \| \leq m.
    \end{align*}
    Then, for all $t>0$
    \begin{align*}
        \P \left(\| \sum_{k=1}^K X_k \| \geq t\right) \leq d \left( \frac{em}{t}\right)^{t/B}.
    \end{align*}
\end{theorem}
Now we are going to derive a bound on $\P \left( \|  \proj H \proj' \| \geq r \right)$ by applying the Matrix Chernoff inequality $3$ times. We first use the randomness of $\proj'$ while holding $\proj$ fixed, then we bound the two resulting terms involving $\proj$. This leads to the following result
\begin{lemma} Let $H \in \R^{d \times K}$ be some matrix. Assume $I,I' \subseteq \mathbb{K}$ - leading to the selector matrices $R, R'$ - are chosen according to the Poisson sampling model with probabilities $p_1, \dots , p_K$ such that $\sum_{i = 1}^K p_i = S$. Further let $W$ denote the corresponding weight matrix. Then, for all $r > 0$
\begin{equation}\label{eq:2}
    \P \left( \| \proj H \proj' \| \geq r \right)\leq  K  \left(e \frac{u^2}{r^2}  \right)^{\frac{r^2}{v^2}} + K \left(e \frac{\|\weights H \weights \|^2}{u^2} \right)^{\frac{u^2}{\| H \weights \|_{\infty,2}^2}} + K \left( e  \frac{\|  W H \|_{2,1}^2}{v^2}  \right) ^{\frac{v^2}{\mu^2}}.
\end{equation}
\end{lemma}
We begin by bounding $\P \left( \|\proj H \proj' \| \geq r \right)$.
\begin{lemma}
Let $H \in \R^{d \times K}$ be some matrix. Assume $I' \subseteq \mathbb{K}$ - leading to the selector matrix $R'$ - is chosen according to the Poisson sampling model with probabilities $p_1, \dots , p_K$ such that $\sum_{i = 1}^K p_i = S$. Further let $W$ denote the corresponding weight matrix. Then, for all $r > 0$
\begin{equation*}
    \P \left( \|\proj H \proj' \| \geq r \right)\leq  K  \left(e \frac{\| \proj H \weights \|^2}{r^2}  \right)^{\frac{r^2}{\|  \proj H \|^2_{2,1}}}.
\end{equation*}
\end{lemma}
\begin{proof}
Using that for any matrix A, $\| A A\transp\| = \| A\transp A \| = \|A\|^2$ we see
\begin{align*}
    \P \left(  \|\proj H \proj' \| >r \right) = \P \left( \|\proj H \proj' \|^2 >r^2  \right) = \P \left(  \| \proj H \proj' H\transp \proj \| >r^2 \right).
\end{align*}
Denoting by $Z_j$ the $j$-th column of $ \proj H$, we get
\begin{equation}\label{eq:4}
    \proj H \proj' H\transp \proj = \sum_{j=1}^K \delta'_j Z_j Z_j\transp.
\end{equation}
Then we have $\|Z_j Z_j\transp \| = \|Z_j \|_2^2 \leq \| \proj H \|_{2,1}^2$
and
\begin{equation*}
    \| \sum_{j=1}^K\E [\delta'_j Z_j Z\transp_j] \| = \| \sum_{j=1}^K p_j Z_j Z\transp_j \| =  \| \proj H \weights \weights H\transp \proj \| = \| \proj H \weights \|^2. 
\end{equation*}
As the right hand side of~\eqref{eq:4} is a sum of independent random variables, an application of the Matrix Chernoff inequality yields the result.
\end{proof}
Now we turn to bounding the two quantities $  \| \proj H \weights \|$ and $\| \proj H \|_{2,1} $ by the same argument as above.
\begin{lemma}
Let $H \in \R^{d \times K}$ be some matrix. Assume $I \subseteq \mathbb{K}$ is chosen according to the Poisson sampling model with probabilities $p_1, \dots , p_K$ such that $\sum_{i = 1}^K p_i = S$. Further let $W$ denote the corresponding weight matrix. Then, for all $u > 0$
\begin{equation*}
    \P \left(  \| \proj H \weights \| >u \right) \leq K \left(e \frac{\|\weights H \weights \|^2}{u^2} \right)^{\frac{u^2}{\|  H \weights \|_{\infty,2}^2}}.
\end{equation*}
\end{lemma}
\begin{proof}
Again using that for any matrix A, $\| A A\transp\| = \| A\transp A \| = \|A\|^2$ we see
\begin{align*}
    \P \left(  \| \proj H \weights \| >u \right) =  \P \left(  \| \proj H \weights \|^2 >u^2\right) = \P \left(  \| \weights H\transp \proj H \weights \| >u^2 \right).
\end{align*}
Now denote by $Y_j$ the $j$-th row of $H \weights$ then we get
\begin{equation}\label{eq:3}
    \weights H\transp \proj H \weights = \sum_{j=1}^K \delta_j Y_j\transp Y_j.
\end{equation}
We have $\|Y_j \transp Y_j \| = \|Y_j \|_2^2 \leq \| H \weights \|_{\infty,2}^2$
and
\begin{equation*}
    \| \sum_{j=1}^K\E [\delta_j Y_j\transp Y_j]\| = \| \sum_{j=1}^K p_j Y_j\transp Y_j \|  =  \| \weights H\transp \diag((p_k)_k)H \weights\| = \| \weights H \weights\|^2.
\end{equation*}
As the right hand side of~\eqref{eq:3} is a sum of independent random variables, an application of the Matrix Chernoff inequality yields the result.
\end{proof}
We now restate and prove Lemma~\ref{lem:1} for the Poisson sampling model. Note that by definition $\| \proj H\transp  \|_{2,1} = \| H  \proj \|_{\infty,2} = \| H_I \|_{\infty,2}$. Recall that by Lemma~\ref{lemma:f}
\[
\P_S \left( \| H _I \|_{\infty,2} \geq v \right) \leq 2 \; \P \left( \| H _I \|_{\infty,2} \geq v \right),
\]
so this result translates immediately to the rejective sampling model.
\begin{lemma}
Let $H \in \R^{d \times K}$ be some matrix. Assume $I \subseteq \mathbb{K}$ is chosen according to the Poisson sampling model with probabilities $p_1, \dots , p_K$ such that $\sum_{i = 1}^K p_i = S$. Further let $W$ denote the corresponding weight matrix. Then, for all $v > 0$
\begin{equation*}
    \P \left( \| H _I \|_{\infty,2} \geq v \right) \leq K \left( e \frac{\| H \weights \|_{\infty,2}^2}{v^2}  \right) ^{\frac{v^2}{\mu^2}}.
\end{equation*}
\end{lemma}
\begin{proof}
We begin by writing $\|H_I\|$ as the maximum of a sum of independent random variables
\begin{equation*}
    \| H_I \|_{\infty,2}^2 = \max_{i\in \{1,...,K \}}\sum_{j = 1}^K \delta_j H_{ij}^2.
\end{equation*}
Now we fix $i \in \{1,..,K\}$ and apply the standard Chernoff inequality
\[
\P\left( \sum_{j = 1}^K \delta_j H_{ji}^2 \geq v^2 \right) \leq \left( e \frac{\| H \weights \|_{\infty,2}^2}{v^2}  \right) ^{\frac{v^2}{\mu^2}}.
\]
Taking a union bound yields the result.
\end{proof}
Finally we can put everything together and prove our main result. The main difficulty lies in picking $v$ and $u$ such as to minimise the probability bound in~\eqref{eq:2}. 
\begin{proof}[Theorem \ref{them:1}] Set
\begin{align*}
\alpha :=  \min\left\{\frac{r^2}{4e^2\|\weights H \|_{2,1}^2},\frac{r^2}{4e^2\| H \weights \|_{\infty,2}^2},\frac{r}{2\mu} \right\}  &&  v^2 := \frac{r^2}{4 \alpha} && u^2 := \frac{r^2}{4 e^2 }.
\end{align*}
Now these definitions and the assumption $r^2 \geq 4 e^4 \| \weights H \weights \|^2$ imply the following 6 inequalities
\begin{align*}
    \frac{u^2}{\|H \weights \|_{\infty,2}^2}= \frac{r^2}{4 e^2 \| H \weights \|_{\infty,2}^2} &\geq \alpha  & e\frac{\| \weights H   \weights \|^2}{u^2} = \frac{4 e^{3} \|  \weights H  \weights \|^2}{r^2}&\leq e^{-1} \\
    \frac{v^2}{\mu^2} = \frac{r^2 }{  4 \alpha \mu^2} &\geq \alpha & e\frac{\| \weights H  \|_{2,1}^2}{v^2} = \frac{4 e \| \weights H \|_{2,1}^2 \alpha}{r^2} &\leq e^{-1} \\
    \frac{r^2}{4 v^2} = \frac{4 r^2 \alpha}{4 r^2} &= \alpha & e \frac{4 u^2}{r^2} = \frac{4 e r^2}{4 e^2 r^2} &= e^{-1}.
\end{align*}
So
\[
\P_S \left ( \| \proj H \proj \| \geq r  \right) \leq 2 \P\left ( \| \proj H\proj \| \geq r  \right) \leq 72 \P \left ( \| \proj H\proj' \| \geq r/2  \right),
\]
together with 
\begin{align*}
    \P \left( \|\proj H\proj' \| \geq r \right)\leq & K \left( \left(e \frac{4u^2}{r^2}  \right)^{\frac{r^2}{4v^2}} +\left(e \frac{\| \weights H \weights \|^2}{u^2} \right)^{\frac{u^2}{\| H  \weights \|_{\infty,2}^2}} + \left( e  \frac{\|\weights H \|_{2,1}^2}{v^2}  \right) ^{\frac{v^2}{\mu^2}} \right)
\end{align*}
shows that
\[
\P_S \left ( \| \proj H \proj \| \geq r  \right) \leq 216 K e^{-\alpha}.
\]
\end{proof}
\begin{remark}
In the published version of Chr\'etien and Darses~\cite{chda12} there is a tiny bug in the proof of Proposition 4.2 in the way the variables $u$ and $v$ are balanced. In particular, for very small $\mu$, inequality 4.17 may be violated. $v^2$ should instead be defined via an equality in 4.15, whereas 4.14 should be an inequality.
\end{remark}
For convenience we restate an easy consequence of Hoeffding's inequality.
\begin{lemma}[Hoeffding]\label{hoeff}
    Let $M \in \R^{K \times S}$ be a matrix and $x \in \R^{S}$ such that $\signop(x) \in \R^{S}$ is an independent Rademacher sequence. Then, for all $t \geq 0$
\[
\P\left(\|M x\|_{\infty } \geq t\right) \leq 2K \exp{\left( -\frac{t^2}{2 \| M \|_{\infty,2}^2\| x\|_{\infty}^2} \right)}.
\]
\end{lemma}
\begin{proof} We apply Hoeffding's inequality to the $k$-th entry of $Mx$, which yields
\[
\P\left( |{\textstyle \sum_j M_{kj}x_j}| \geq t\right) \leq 2 \exp{\left( -\frac{t^2}{2\sum_j M^2_{kj}x^2_j} \right)} \leq 2 \exp{\left( -\frac{t^2}{2\|x\|_\infty^2 \|M^k\|_2^2}\right)}.
\]
The statement follows using a union bound and the identity $\| M \|_{\infty,2} = \max_k \|M^k\|_2$.
\end{proof}

\section{Sensing matrices}
\begin{lemma}[Thresholding with sensing matrix]\label{lem:thr_sens}
Assume that the signals follow the model in \eqref{signal_model}, where the support $I \subseteq\mathbb{K}$ is chosen according to the rejective sampling model with probabilities $p_1, \dots , p_K$ such that $\sum_{i = 1}^K p_i = S$. Further let $W$ denote the corresponding weight matrix and denote by $H := \pdico\transp \dico - \mathbb{I}$ the hollow cross-Gram matrix. If
\begin{align*}
    \| H \|^2_{\infty,1} \leq \frac{\| c \|_{\min}^2}{8 \| c \|_{\max}^2 \log( 4K/\varepsilon)}, \quad \quad \text{and} \quad \quad
    \| H \weights \|^2_{\infty,2} \leq \frac{\| c \|_{\min}^2}{8 e^2 \| c \|_{\max}^2\log(4K/\varepsilon)},
\end{align*}
then thresholding with sensing dictionary $\pdico$ recovers the support with probability at least $1- \varepsilon$.
\end{lemma}
\begin{proof}
Now by definition of the algorithm, thresholding recovers the full support if
\[
\| \pdico_{I^c}\transp y \|_{\max} <  \| \pdico_I\transp y\|_{\min}.
\]
Repeating the steps from the proof of Theorem~\ref{them:thresh} with the obvious changes we obtain the result.
\end{proof}

\begin{lemma}[OMP with sensing matrix]\label{lem:omp_sens}
Assume that the signals follow the model in \eqref{signal_model}, where the support $I \subseteq\mathbb{K}$ is chosen according to the rejective sampling model with probabilities $p_1, \dots , p_K$ such that $\sum_{i = 1}^K p_i = S$. Further let $W$ denote the corresponding weight matrix. Let $\pdico$ be a sensing matrix and assume the hollow Gram-matrix $H = \dico\transp \dico - \mathbb{I}$ satisfies $\| \weights H \weights\|_{2,2} \leq \frac{1}{4 e^2}$. If
\begin{align*}
\| H \weights \|_{\infty,2}^2 &\leq   \frac{1}{16e^2\log(216 K/\varepsilon)}  \\ 
\| H  \|_{\infty,1} &\leq \frac{1}{4\log(218K/\varepsilon)} \\
\| (\pdico \transp \dico - \mathbb{I}) \weights \|_{\infty,2}^2 &\leq   \min_{L \subseteq \{ 1, \dots, S \} }\frac{\|c_L \|^2_{\infty}}{16 e^2\|c_L \|_2^2} \\
\| \pdico \transp \dico - \mathbb{I}  \|_{\infty,1}  &\leq \min_{L \subseteq \{ 1, \dots, S \} }\frac{\|c_L \|_{\infty}}{4\|c_L \|_2 \sqrt{\log(218 K/\varepsilon)}},
\end{align*} 
then OMP with sensing matrix $\pdico$ recovers the correct support with probability at least $1 - \varepsilon$.
\end{lemma}
\begin{proof}
Set $L := I \setminus J$. By definition, OMP finds another correct atom in the next step if
\begin{equation*}
\| \pdico_{I^c }\transp( \dico_{L} x_{L} -\dico_{J} (\dico_{J}\transp\dico_{J})^{-1} \dico_{J}\transp \dico_{L} x_{L})  \|_{\infty} < \| \pdico_{L}\transp (\dico_{L} x_{L} -\dico_{J} (\dico_{J}\transp\dico_{J})^{-1} \dico_{J}\transp \dico_{L} x_{L} ) \|_{\infty}.
\end{equation*}
Repeating the steps from the proof of Theorem~\ref{them:omp} with the obvious changes we obtain the result.
\end{proof}



\begin{thebibliography}{10}

\bibitem{adhaporo17}
{\sc B.~Adcock, A.~Hansen, C.~Poon, and B.~Roman}, {\em Breaking the coherence
  barrier: A new theory for compressed sensing.}, Forum of Mathematics, Sigma,
  5 (2017), p.~e4, \url{https://doi.org/10.1017/fms.2016.32}.

\bibitem{botz87}
{\sc J.~Bourgain and L.~Tzafriri}, {\em Invertibility of ``large'' submatrices
  with applications to the geometry of banach spaces and harmonic analysis.},
  Israel J. Math, 57 (1987), pp.~137--224.

\bibitem{boyer13}
{\sc C.~Boyer, P.~Weiss, and J.~Bigot}, {\em An algorithm for variable density
  sampling with block-constrained acquisition}, SIAM Journal on Imaging
  Sciences [electronic only], 7 (2013),
  \url{https://doi.org/10.1137/130941560}.

\bibitem{candes2009}
{\sc E.~{C}and{\`e}s and Y.~Plan}, {\em Near-ideal model selection by l1
  minimization}, The Annals of Statistics, 37 (2009), pp.~2145--2177.

\bibitem{candes11}
{\sc E.~{C}and{\`e}s and Y.~{Plan}}, {\em A probabilistic and ripless theory of
  compressed sensing}, IEEE Transactions on Information Theory, 57 (2011),
  pp.~7235--7254, \url{https://doi.org/10.1109/TIT.2011.2161794}.

\bibitem{carota06}
{\sc E.~{C}and{\`e}s, J.~{R}omberg, and T.~{T}ao}, {\em {R}obust uncertainty
  principles: exact signal reconstruction from highly incomplete frequency
  information}, {IEEE} {T}ransactions on {I}nformation {T}heory, 52 (2006),
  pp.~489--509.

\bibitem{weiss13_2}
{\sc N.~Chauffert, P.~Ciuciu, J.~Kahn, and P.~Weiss}, {\em Variable density
  sampling with continuous trajectories}, SIAM Journal on Imaging Sciences, 7
  (2013), \url{https://doi.org/10.1137/130946642}.

\bibitem{chda12}
{\sc S.~Chr\'etien and S.~Darses}, {\em Invertibility of random submatrices via
  tail-decoupling and matrix {C}hernoff inequality}, Statistics and Probability
  Letters, 82 (2012), pp.~1479--1487.

\bibitem{DecouplingBook}
{\sc V.~De~la Pe{\~n}a and E.~Gin{\'e}}, {\em Decoupling: From Dependence to
  Independence}, Springer New York, 1999.

\bibitem{do06cs}
{\sc D.~Donoho}, {\em {C}ompressed sensing}, {IEEE} {T}ransactions on
  {I}nformation {T}heory, 52 (2006), pp.~1289--1306.

\bibitem{Fuchs2004}
{\sc J.~Fuchs}, {\em On sparse representations in arbitrary redundant bases},
  IEEE Transactions on Information Theory, 50 (2004), pp.~1341--1344.

\bibitem{hajek1964}
{\sc J.~Hajek}, {\em Asymptotic theory of rejective sampling with varying
  probabilities from a finite population}, Annals of Mathematical Statistics,
  35 (1964), pp.~1491--1523.

\bibitem{jogdeo1968}
{\sc K.~Jogdeo and S.~M. Samuels}, {\em Monotone convergence of binomial
  probabilities and a generalization of ramanujan's equation}, Ann. Math.
  Statist., 39 (1968), pp.~1191--1195.

\bibitem{krwa12}
{\sc F.~{Krahmer} and R.~{Ward}}, {\em Stable and robust sampling strategies
  for compressive imaging}, IEEE Transactions on Image Processing, 23 (2014),
  pp.~612--622, \url{https://doi.org/10.1109/TIP.2013.2288004}.

\bibitem{pairs_dictionary_recovery}
{\sc P.~{Kuppinger}, G.~{Durisi}, and H.~{Bolcskei}}, {\em Uncertainty
  relations and sparse signal recovery for pairs of general signal sets}, IEEE
  Transactions on Information Theory, 58 (2012), pp.~263--277.

\bibitem{Vander11}
{\sc G.~{Puy}, P.~{Vandergheynst}, and Y.~{Wiaux}}, {\em On variable density
  compressive sampling}, IEEE Signal Processing Letters, 18 (2011),
  pp.~595--598, \url{https://doi.org/10.1109/LSP.2011.2163712}.

\bibitem{Randall2009}
{\sc P.~Randall}, {\em Sparse recovery via convex optimization}, ph.d. thesis,
  n.4349, California Institute of Technology, 2009.

\bibitem{sc18omp}
{\sc K.~Schnass}, {\em Average performance of {O}rthogonal {M}atching {P}ursuit
  ({OMP)} for sparse approximation}, IEEE Signal Processing Letters, 25 (2018),
  pp.~1865--1869.

\bibitem{scva07}
{\sc K.~Schnass and P.~Vandergheynst}, {\em Average performance analysis for
  thresholding}, IEEE Signal Processing Letters, 14 (2007), pp.~828--831.

\bibitem{scva08}
{\sc K.~Schnass and P.~Vandergheynst}, {\em {D}ictionary preconditioning for
  greedy algorithms}, IEEE Transactions on Signal Processing, 56 (2008),
  pp.~1994--2002.

\bibitem{troppl1}
{\sc J.~{Tropp}}, {\em Recovery of short, complex linear combinations via l1
  minimization}, IEEE Transactions on Information Theory, 51 (2005),
  pp.~1568--1570.

\bibitem{tropp08}
{\sc J.~Tropp}, {\em Norms of random submatrices and sparse approximation},
  Comptes Rendus Mathematique, 346 (2008), pp.~1271--1274.

\bibitem{tr08}
{\sc J.~{T}ropp}, {\em On the conditioning of random subdictionaries}, Applied
  and Computational Harmonic Analysis, 25 (2008).

\bibitem{tr12}
{\sc J.~{T}ropp}, {\em User-friendly tail bounds for sums of random matrices},
  Foundations of Computational Mathematics, 12 (2012), pp.~389--434.

\end{thebibliography}
\end{document}